\documentclass[11pt]{article}\textwidth 160mm\textheight 235mm
\usepackage{amsfonts}
\usepackage{amssymb}
\usepackage{graphicx}
\usepackage{amsmath}
\usepackage{amsthm}
\usepackage{enumitem}
\usepackage{tikz}
\usepackage{mathrsfs}
\usepackage{cancel}
\usepackage{todonotes}
\usepackage{cases}
\usetikzlibrary{matrix}
\usetikzlibrary{arrows,shapes}
\usepackage{cite}
\usepackage{bm}
\usepackage{hyperref}
\hypersetup{colorlinks=true, urlcolor=blue}
\expandafter\let\expandafter\oldproof\csname\string\proof\endcsname
\let\oldendproof\endproof
\renewenvironment{proof}[1][\proofname]{%
  \oldproof[\ttfamily \scshape \bf #1. ]%
}{\oldendproof}

\def\O{{\cal O}}
\def\C{{\cal C}}
\def\B{\mathbb{B}}
\def\R{{\rm I\!R}}
\def\oR{\overline{\R}}
\def\N{{\rm I\!N}}
\def\oi{{i_0}}
\def\oj{{j_0}}
\def\ox{\bar{x}}
\def\oy{\bar{y}}
\def\oz{\bar{z}}
\def\ov{\bar{v}}

\def \b{{\}_{k\in\N}}}

\def\xk{x^k}
\def\yk{y^k}
\def\zk{z^k}
\def\lmk{{\lambda^k}}
\def\wk{w^k}

\def\what{\widehat}

\def\emp{\emptyset}

\def\tto{\rightrightarrows}

\def\tto{\rightrightarrows}

\def\sub{\partial}
\def\ra{\rangle}
\def\la{\langle}

\def\lm{\lambda}
\def\olm{{\bar\lambda}}

\def\dd{\delta}
\def\al{\alpha}

\def\otau{{\bar\tau}}
\def\osigma{{\bar\sigma}}
\def\ph{\varphi}

\def\toset_#1{\xrightarrow{#1}}
\DeclareMathOperator*{\mini}{minimize\;}
\DeclareMathOperator*{\argmin}{argmin}
\def\prox{{\rm prox}}
\def\d{{\rm d}}
\def\dist{{\rm dist}}

\def\rge{{\rm rge\,}}
\def\co{{\rm co}\,}
\def\cone{{\rm cone}\,}
\def\ri{{\rm ri}\,}

\def\inte{{\rm int}\,}
\def\gph{{\rm gph}\,}
\def\epi{{\rm epi}\,}
\def\dim{{\rm dim}\,}
\def\dom{{\rm dom}\,}

\def\ker{{\rm ker}\,}

\def\Kg{{K_g(\oz, \olm)}}

\def\Kgzv{{K_g(z,\lm)}}
\def\K{{\overline{  K}}}
\def\rN{{\what{N}}}

\begin{document}
\vspace*{0.5in}
\begin{center}
{\bf ROLE OF SUBGRADIENTS IN VARIATIONAL ANALYSIS OF POLYHEDRAL FUNCTIONS}\\[2ex]
NGUYEN T. V. HANG\footnote{Institute of Mathematics, Vietnam Academy of Science and Technology, Hanoi, Vietnam (ntvhang@math.ac.vn). Research of this author is partially supported by Vietnam Academy of Science and Technology under grant CTTH00.01/22-23.}, WOOSUK JUNG\footnote{ 
 Department of Mathematics, Miami University, Oxford, OH 45065, USA (jungw2@miamioh.edu).}, and  M. EBRAHIM SARABI\footnote{Corresponding author, Department of Mathematics, Miami University, Oxford, OH 45065, USA (sarabim@miamioh.edu). Research of this    author is partially supported by the U.S. National Science Foundation  under the grant DMS 2108546.}
\end{center}
\vspace*{0.05in}

\small{\bf Abstract.} Understanding the role that subgradients play in various second-order variational analysis constructions can 
help us   uncover new  properties of important classes of functions in variational analysis. Focusing mainly on the behavior 
of the second subderivative and subgradient proto-derivative of polyhedral functions, functions with polyhedral epigraphs, we  demonstrate that choosing 
the underlying subgradient, utilized in the definitions of these concepts, from  the relative interior of the subdifferential  of polyhedral functions ensures stronger 
second-order variational properties such as strict twice epi-differentiability and strict subgradient proto-differentiability. This allows us to characterize continuous differentiability of the proximal 
mapping and twice continuous differentiability of the Moreau envelope of polyhedral functions. We close the paper with proving the equivalence of metric regularity 
and strong metric regularity of a class of generalized equations at their  nondegenerate solutions. 
 \\[1ex]

\newtheorem{Theorem}{Theorem}[section]
\newtheorem{Proposition}[Theorem]{Proposition}
\newtheorem{Lemma}[Theorem]{Lemma}
\newtheorem{Corollary}[Theorem]{Corollary}
\newtheorem{Definition}[Theorem]{Definition}
\numberwithin{equation}{section}

\theoremstyle{definition}
\newtheorem{Example}[Theorem]{Example}
\newtheorem{Remark}[Theorem]{Remark}

\renewcommand{\thefootnote}{\fnsymbol{footnote}}
\setcounter{footnote}{0}
\normalsize

\section{Introduction}
Second-order variational constructions such as  the second subderivative and subgradient proto-derivative play an important role in parametric optimization and convergence analysis of 
important numerical algorithms such as   the augmented Lagrangian method \cite{hs, r22}. Given an extended-real-valued function 
$f:\R^n\to  \oR:=[-\infty,\infty]$, these second-order variational constructions are defined at a point $(x,v)$ in the graph of the subgradient mapping of $f$. One   may wonder what impacts the selection 
of the subgradient $v$ can have in these constructions. 
Our main goal in this paper is to study the underlying role that the selection of a subgradient  can play for such constructions. 
To achieve this goal, we focus mainly on a particular class of convex functions called polyhedral functions. Recall that 
 a proper function $g:\R^m\to \oR$ is called polyhedral if  its epigraph is a polyhedral convex set. 
Given  $\oz\in \R^m$ with $g(\oz)$  finite, consider a subgradient $\olm\in   \sub g(\oz)$. It is well-known (cf. \cite[Proposition~13.9]{rw}) that the polyhedral 
function $g$ is twice epi-differentiable at $\oz$ for $\olm$, that the subgradient mapping $\sub g$ is proto-differentiable at $\oz$ for $\olm$, and that its proximal mapping 
  is directionally differentiable at $\oz+\olm$; see Sections~\ref{sec03} and \ref{sec4} for the definitions of these concepts. 

Should we expect further properties 
if, in addition, we assume that $\olm\in \ri \sub g(\oz)$? This is the main question that we are going to investigate for this class of convex functions. Note that such a relative interior 
condition has been utilized in several studies related to different numerical  methods, including the partial smoothness \cite{l03} and the ${\cal U}$-Lagrangian of
convex functions \cite{los}.  Therefore  understanding 
 the role that is  played by  this condition in second-order variational analysis can lead to stronger stability properties for important classes of  functions  in variational analysis  as demonstrated in Sections \ref{sec4} and \ref{sec05}. 
 
 Following this question, we  uncover new second-order variational properties of polyhedral functions, including strict twice epi-differentiability and strict subgradient proto-differentiability,  under this extra assumption.
These findings allow us to achieve a useful characterization of  continuous differentiability of the proximal mapping and twice continuous differentiability of the Moreau envelope of polyhedral functions. As an important 
application, we turn to study stability properties of the solution mapping to the generalized equation 
  \begin{equation}\label{vs1}
0\in \psi(x)+\partial g(x),
\end{equation}
where $\psi:\R^m\to\R^m$ is  a ${\cal C}^1$ mapping and $g: \R^m \to \oR $ is a polyhedral function. 
Our interest is mainly in examining the relationship between metric regularity and strong metric regularity for \eqref{vs1}. 
The seminal work of Donchev and Rockafellar in \cite{dr96} demonstrated that these properties are equivalent for \eqref{vs1} when the polyhedral function $g$ is the indicator function 
of a polyhedral convex set. Employing our new developments under the relative interior condition, we are going to show that if $\ox$ is 
a {\em nondegenerate}  solution to the generalized equation \eqref{vs1}, meaning that it satisfies the condition 
\begin{equation}\label{mr1}
-\psi(\ox) \in \ri \partial g(\ox),
\end{equation}
then under some verifiable assumptions   the solution mapping to the canonical perturbation 
of \eqref{vs1} has a Lipschitz continuous single-valued localization, which is continuously differentiable. 
The latter smoothness of a localization of solution mappings of 
generalized equations resembles a similar conclusion from the classical inverse mapping theorem.  

The rest of the paper is organized as follows. Section~\ref{sec02}  contains definitions of important concepts, used in this paper. We also establish some properties 
of polyhedral functions. Section~\ref{sec03} begins with a new proof of the reduction lemma for polyhedral functions and then we present its important consequences in various second-order 
variational constructions. In particular, we show that under the relative interior condition, the   subgradient mappings of polyhedral functions are strictly proto-differentiable.
 Section~\ref{sec4} is devoted to study strict twice epi-differentiablity of polyhedral functions. As an important consequence, we characterize continuous differentiability of the proximal 
 mapping and twice continuous differentiability of the Moreau envelope of polyhedral functions. The final section, Section~\ref{sec05}, concerns the equivalence of metric regularity and strong metric regularity 
 for the generalized equation \eqref{vs1}. Using this equivalence and \eqref{mr1}, we present sufficient conditions for a smooth single-valued localization of the solution mapping 
 to the canonical perturbation of \eqref{vs1}. 

\section{Notation and Preliminary Results}\label{sec02} 
In what follows,    we 
 denote by $\B$ the closed unit ball in the space in question and by $\B_r(x):=x+r\B$ the closed ball centered at $x$ with radius $r>0$. 
  In the  product space $\R^n\times \R^m$, we use the norm $\|(w,u)\|=\sqrt{\|w\|^2+\|u\|^2}$ for any $(w,u)\in \R^n\times \R^m$.
 Given a nonempty set $C\subset\R^n$, the symbols $\inte C$,  $\ri C$, $\cone C$, and $\co C$ signify its interior, relative interior, conic hull, and convex hull, respectively. 
 For any set $C$ in $\R^n$, its indicator function is defined by $\dd_C(x)=0$ for $x\in C$ and $\dd_C(x)=\infty$ otherwise. We denote
 by $P_C$ the projection mapping onto $C$ and  by $\dist(x,C)$  the distance between $x\in \R^n$ and a set $C$.
 For a vector $w\in \R^n$, the subspace $\{tw |\, t\in \R\}$ is denoted by $[w]$. The domain and range of a set-valued mapping $F:\R^n\tto\R^m$ are defined, respectively, by  
$ \dom F:= \{x\in\R^n\big|\;F(x)\ne\emp \}$ and $\rge F=\{u\in \R^m|\; \exists \,  w\in \R^n\;\,\mbox{with}\;\; u\in  F(w) \}$.

In this paper, the convergence of a family of sets is always understood in the sense of Painlev\'e-Kuratowski (cf. \cite[Definition~4.1]{rw}). 
This means that the inner limit set of a parameterized family of sets $\{C^t\}_{t>0}$ in $\R^d$, denoted $\liminf_{t\searrow 0} C^t$, is the set of points $x$ such that for every
sequence $t_k \searrow 0$,  $x$ is the limit of a sequence of points $x^{t_k}\in C^{t_k}$. The outer limit set of this family of sets, 
denoted $\limsup_{t\searrow 0} C^t$, is the set of points $x$ such that there exist  sequences
 $t_k \searrow 0$ and $x^{t_k}\in C^{t_k}$ such that  $x^{t_k}\to x$ as $k\to \infty$. A sequence $\{f^k\b$ of functions $f^k:\R^n\to \oR$ is said to {\em epi-converge} to a function $f:\R^n\to \oR$ if we have $\epi f^k\to \epi f$ as $k\to \infty$,
 where $\epi f=\{(x,\al)\in \R^n\times \R|\, f(x)\le \al\}$ is   the epigraph of $f$;
 see \cite[Definition~7.1]{rw} for more details on the epi-convergence of 
a sequence of extended-real-valued functions.  

Given a nonempty set $\Omega\subset\R^n$ with $\ox\in \Omega$, the {  tangent cone}  to $\Omega$ at $\ox$, denoted $T_\Omega(\ox)$,  is defined  by
\begin{equation}\label{tan1}
T_\Omega(\ox) = \limsup_{t\searrow 0} \frac{\Omega - \ox}{t}.
\end{equation}
The  regular/Fr\'{e}chet normal cone $\rN_\Omega(\ox)$ to $\Omega$ at $\ox$ is defined by
 $\rN_\Omega(\ox) = T_\Omega(\ox)^*$, the polar of the tangent cone \eqref{tan1}. The  (limiting/Mordukhovich) normal cone $N_\Omega(\ox)$ to $\Omega$ at $\ox$ is
 the set of all vectors $\ov\in \R^n$ for which there exist sequences  $\{x^k\b$ and  $\{v^k\b$ with $v^k\in \rN_\Omega( x^k)$ such that 
$(x^k,v^k)\to (\ox,\ov)$. When $\Omega$ is convex, both normal cones boil down to that of convex analysis.  
Given a function $f:\R^n \to \oR$ and a point $\ox\in\R^n$ with $f(\ox)$ finite, the subderivative function $\d f(\ox)\colon\R^n\to\oR$ is defined by
\begin{equation*}\label{fsud}
\d f(\ox)(w)=\liminf_{\substack{
   t\searrow 0 \\
  w'\to w
  }} {\frac{f(\ox+tw')-f(\ox)}{t}}.
\end{equation*}
When $f$ is convex, its  subdifferential    at $\ox$, denoted by $\sub f(\ox)$,  is understood in the sense of convex analysis (cf. \cite[page~214]{Roc70}), namely $v\in \sub f(\ox)$ if   $f(x)\ge  f(\ox)+\la v,x-\ox\ra $ for any $x\in \R^n$.

As pointed earlier,  a proper   function $g:\R^m\to \oR$ is called {\em polyhedral}  
if $\epi g$ is a polyhedral convex set.  According to \cite[Theorem~2.49]{rw}, this class of convex functions enjoys the  representation
\begin{equation*}\label{polyf}
g(z) = \begin{cases}
\max\limits_{j\in J}\big\{\la a^j, z \ra - \alpha_j\big\} \quad& \text{if } z \in\dom g,\\
\infty & \text{otherwise,}
\end{cases}
\end{equation*}
where $J= \{1,\ldots, l\}$ for some $l\in \N$, $a^j\in \R^m$ and $\alpha_j\in \R$  for all $j\in J$, and where  $\dom g=\{z\in \R^m|\, g(z)<\infty\}$ is a polyhedral convex set with the representation 
\begin{equation}\label{dom}
\dom g = \big\{z\in\R^m\, \big|\, \la b^i, z \ra \leq \beta_i,  i\in I=\{1, \ldots, s\} \big\},
\end{equation}
where $s\in \N$, and $b^i \in \R^m$ and $\beta_i\in \R$ for all $i\in I$. Thus, we can equivalently express a polyhedral function  $g$ as
\begin{equation}\label{polyfunc}
g(z) = \max_{j\in J}\big\{\la a^j, z \ra - \alpha_j\big\} + \delta_{\dom g}(z), \quad z\in \R^m.
\end{equation}
It was observed in \cite[Proposition~3.2]{MoS16} that $\dom g$ can be expressed as the finite  union of the polyhedral convex sets $C_j, j\in J$, defined by
\begin{equation*}\label{Cj}
C_j =\big\{z\in \dom g\, \big|\, g(z) = \la a^j, z\ra -\alpha_j\big\} =\big\{z \in \dom g\, \big|\, \la a^i -a^j, z\ra  \leq \alpha_i -\alpha_j, i\in J\big\}.
\end{equation*}
Pick $\oz \in\dom g$ and  define the sets of active indices  at $\oz$ corresponding to the representation \eqref{dom}    and  {to the expression of $\dom g$ via the finite union of $C_j$'s} by
\begin{equation}\label{indset4}
I(\oz)= \big\{i\in I\, \big|\, \la b^i, \oz\ra = \beta_i\big\}\quad\textrm{ and }\quad J(\oz)= \big\{ j\in J\, \big|\, \oz \in C_j\big\}.
\end{equation}
These sets allow us to conclude from  \cite[Proposition~3.3]{MoS16} that  the subdifferential of $g$ at $\oz$ can be calculated as
\begin{equation}\label{subg}
\partial g (\oz) = \co \big\{a^j \, \big|\,  j\in J(\oz)\big\}+\cone \big\{b^i\, \big|\, i\in I(\oz)\big\},
\end{equation}
which tells us that  any $ \olm \in   \partial g(\oz)$ can be written as  
 {\begin{equation}\label{decomp}
\olm = \sum_{j\in J(\oz)} \osigma_j a^j + \sum_{i\in I(\oz)} \otau_ib^i\quad
 \textrm{ with }\; \; \osigma_j, \otau_i \geq 0\; \textrm{ and }\;  \sum_{j\in J(\oz)} \osigma_j =1.
\end{equation}
For simplicity, we denote by $\osigma \in \R^l$ and $\otau \in \R^s$ the vectors with components $\osigma_j$ and $\otau_i$, respectively, where $\osigma_j$  is taken from \eqref{decomp} if   $j\in J(\oz)$ and  $\osigma_j = 0$ if $j\not\in J(\oz)$ and likewise, 
$\otau_i$  is taken from \eqref{decomp} if  $i\in I(\oz)$ and $\otau_i =  0$ if $i\notin I(\oz)$. Pick the given representation of $\olm$ in \eqref{decomp} and define the sets of positive coefficients at $\oz$ for $\osigma$ and $\otau$, respectively, by
\begin{equation}\label{indset}
J_+(\oz, \osigma) = \big\{j\in J(\oz)\, \big|\, \osigma_j >0\big\}\; \; \textrm{ and }\; \; I_+(\oz, \otau)=\big\{i\in I(\oz)\, \big|\, \otau_i>0\big\}.
\end{equation}
Note that the representation of $\olm$ in \eqref{decomp} is not unique. We demonstrate in the next result that  all these representations of $\olm$ enjoy an interesting  property, important for our proof of the reduction lemma in the next section. 
 A similar result was established in \cite[Theorem~3.4]{MoS16} but the neighborhood obtained therein depends on a chosen decomposition of 
 the subgradient $\olm$  in \eqref{decomp}. Next, we show that such  a neighborhood can be chosen to be independent  of  a given representation  of $\olm$.}
 We also simplify the proof presented in \cite{MoS16} significantly. 
\begin{Lemma}\label{lem:indset} Assume that $g:\R^m\to \oR$ is a polyhedral function and that  $(\oz,\olm)\in\gph\sub g$. Then there exists $r>0$ such that for any decomposition of $\olm$ as 
\eqref{decomp} and any $(z, \lm) \in \big(\gph \partial g\big)\cap \B_r(\oz, \olm)$, we have
\begin{equation}\label{indset7}
 J_+(\oz, \osigma) \subset J(z)\; \; \textrm{ and }\; \; I_+(\oz, \otau) \subset I(z). 
\end{equation}
\end{Lemma}

\begin{proof}
Suppose by contradiction that for each $k\in \N$, there are a decomposition of $\olm  $ as
\begin{equation}\label{indset1}
\olm =  \sum_{j\in J(\oz)}\osigma_j^ka^j+ \sum_{i\in I(\oz)}\otau_i^kb^i\; \; \text{ with }\; \;  \osigma_j^k,\otau_i^k\ge 0\; \text{ and }\; \sum_{j\in J(\oz)}\osigma_j^k=1,
\end{equation}
and $(\zk,\lmk)\in\gph\sub g$ such that $(\zk,\lmk)\to(\oz,\olm)$ as $k\to\infty$ with $J_+^k\not\subset J(\zk)$ or $I_+^k\not\subset I(\zk)$, where $J_+^k:= J_+(\oz,\osigma^k)$ 
and $I_+^k:=I_+(\oz,\otau^k)$ are defined via   \eqref{indset}. Since $\zk \to \oz$ as $k\to\infty$, the inclusions  $J(\zk)\subset J(\oz)$
 and $I(\zk)\subset I(\oz)$ hold for all $k$ sufficiently large. Passing to a subsequence if necessary, we can assume that there exist subsets $\bar J\subset J(\oz)$ and $\bar I \subset I(\oz)$ such that
\begin{equation*}
J(\zk)=\bar J \quad \text{ and }\quad I(\zk)=\bar I \quad\text{ for all }k\in \N,
\end{equation*}
which, together with \eqref{subg}, lead us to 
\begin{equation*}
\partial g(\zk) = \co\big\{a^j|\, j\in \bar J\big\}+\cone\big\{b^i|\, i\in \bar I\big\}=:\Omega\quad\textrm{ for all } k\in \N.
\end{equation*}
Since $\lmk \in \partial g(\zk) = \Omega$ and $\lmk\to \olm$ as $k\to\infty$, we arrive at  $\olm\in\Omega =\partial g(\zk)$ for all  $k$ sufficiently large. Fix such a  $k\in \N$ and  deduce from $\olm \in \partial g(\zk)$ that
\begin{equation}\label{indset3}
\la \olm, \oz-\zk\ra \leq g(\oz)-g(\zk).
\end{equation}
By the decomposition in \eqref{indset1}, we obtain  
\begin{align}\label{indset2}
\la \olm, \zk - \oz\ra &= \sum_{j\in J(\oz)} \osigma_j^k\la a^j, \zk-\oz\ra +\sum_{i\in I(\oz)}\otau_i^k\la b^i, \zk - \oz\ra\nonumber\\
&= \sum_{j\in J(\oz)} \osigma_j^k\big(\la a^j, \zk\ra - \alpha_j - g(\oz)\big) +\sum_{i\in  I(\oz)}\otau_i^k \big(\la b^i, \zk\ra - \beta_i\big),\nonumber\\
&\leq \sum_{j\in J(\oz)} \osigma_j^k\big(g(\zk) - g(\oz)\big),
\end{align}
where  the second equality results from  \eqref{indset4} and the last inequality comes from  the fact that $\zk\in \dom g$, combined with  \eqref{dom}--\eqref{polyfunc}. If $J_+^k\not\subset J(\zk)$,  there exists   $\oj\in  J_+^k$ such that $\zk\notin C_\oj$, meaning that 
\begin{equation*}\label{indset5}
\osigma_\oj^k \big(\la a^\oj, \zk\ra - \alpha_\oj - g(\oz)\big) < \osigma_\oj^k \big(g(\zk) - g(\oz)\big).
\end{equation*}
If $I_+^k\not\subset I(\zk)$, we find  $\oi\in I_+^k$ such that $\la b^\oi, \zk\ra <\beta_\oi$, which implies that 
\begin{equation*}\label{indset6}
\otau_\oi^k \big(\la b^\oi, \zk\ra - \beta_\oi\big) < 0,
\end{equation*}
which tells us that  in both cases the last inequality in \eqref{indset2} is strict. Thus, our assumption that either $J_+^k\not\subset J(\zk)$ or $I_+^k\not\subset I(\zk)$, together  with the last condition in \eqref{indset1}, yields
\begin{equation*}
\la \olm, \zk - \oz\ra < \sum_{j\in J(\oz)} \osigma_j^k\big(g(\zk) - g(\oz)\big) = g(\zk) - g(\oz),
\end{equation*}
which contradicts \eqref{indset3} and therefore completes the proof.
\end{proof}

We finish this section by recording some  {first-order} variational properties of polyhedral functions, important for our developments in this paper.
Recall from \cite[Exercise~6.47]{rw} that for any polyhedral convex set $C\subset \R^n$ and $\ox\in C$, we can find  a neighborhood $\O$ of $0\in \R^n$ for which we have 
\begin{equation}\label{tanrep}
T_C(\ox)\cap \O=(C-\ox)\cap \O.
\end{equation}

\begin{Proposition}[first-order variational properties] \label{fopg}
Assume that $g:\R^m\to \oR$ is a polyhedral function and $\oz\in \dom g$. Then the following properties  hold.
\begin{itemize}[noitemsep,topsep=2pt]
\item [ \rm {(a)}] The domain of the subderivative function $\d g(\oz)$ can be calculated by 
\begin{equation*}\label{domdg}
\dom \d g(\oz) = T_{\dom g}(\oz) = \bigcup_{j\in J(\oz)} T_{C_j}(\oz).
\end{equation*}
\item [ \rm {(b)}] If $w\in T_{C_j}(\oz)$ for some $j\in J(\oz)$, then we have 
$\d g(\oz) (w) =\la a^j, w\ra$. 
Moreover, there exists $r> 0$ such that  any  $w\in T_{\dom  g}(\oz)\cap \B_r(0)$, the representation 
\begin{equation}\label{approx}
g(\oz + w) = g(\oz) +\d g(\oz)(w)
\end{equation}
holds.
\end{itemize}

\end{Proposition}

\begin{proof} The first equality in (a) was established in \cite[Proposition~10.21]{rw}. The second equality results immediately from the fact that $\dom g=\bigcup_{j\in J} C_j$.
The first claim in  (b) can be found again in \cite[Proposition~10.21]{rw}. To prove  \eqref{approx}, recall that for any $j\in J(\oz)$,    $C_j$ is  a polyhedral convex set. Employing \eqref{tanrep} for these sets, we can find $r>0$ such that
\begin{equation*}
T_{C_j}(\oz)\cap \B_r(0)=(C_j-\oz)\cap \B_r(0)\quad\textrm{ for all }\; j\in J(\oz).
\end{equation*}
Pick any $w \in T_{\dom g}(\oz)\cap \B_r(0)$,  and conclude from  (a) that $w\in T_{C_\oj}(\oz)\cap \B_r(0)$ for some $\oj \in J(\oz)$ and from the first claim in (b) that $\d g(\oz)(w) =\la a^\oj, w\ra$. Thus, we get    $\oz + w \in C_\oj$, which, together with the definition of $C_{j_0}$ and \eqref{indset4}, 
leads us to 
\begin{equation*}
g(\oz+ w) = \la a^\oj, \oz+w\ra-\alpha_\oj = g(\oz) +\la a^\oj, w\ra = g(\oz)+\d g(\oz)(w),
\end{equation*}
and hence completes the proof.
\end{proof}


  
Given a function $f:\R^n \to \oR$ and a point $\ox\in\R^n$ with $f(\ox)$ finite, the critical cone of $f$ at $\ox$ for $\bar v$ with $\bar v\in   \sub f(\ox)$ is defined by 
\begin{equation*}\label{cricone}
{K_f}(\ox,\bar v)=\big\{w\in \R^n\,\big|\,\la\bar v,w\ra=\d f(\ox)(w)\big\}.
\end{equation*}
If, in addition, $f$ is convex and $\partial f(\ox)\neq\emp$, it follows from   \cite[Theorem~8.30]{rw} that  its subderivative function is  the support function of $\partial f(\ox)$, that is
\begin{equation*}\label{fsud2}
\d f(\ox)(w) = \sup\big\{\la v, w\ra\, \big|\, v\in \partial f(\ox)\big\},
\end{equation*}
which in turn allows us to equivalently describe the critical cone ${K_f}(\ox,\bar v)$ as 
\begin{equation}\label{cricone1}
K_f(\ox, \bar v) = N_{\partial f(\ox)}(\bar v).
\end{equation}

\begin{Proposition} \label{nonde1} Assume that $g:\R^m\to \oR$ is a polyhedral function and $(\oz, \olm)\in \gph \partial g$. Then the following properties are equivalent:
\begin{itemize}[noitemsep,topsep=2pt]
\item [ \rm {(a)}]  $\olm\in \ri \sub g(\oz)$;
\item [ \rm {(b)}]  the critical cone $K_g(\oz,\olm)$ is a linear subspace.
\end{itemize}
\end{Proposition}
\begin{proof} By  \eqref{cricone1}, we have $\Kg=N_{\sub g(\oz)}(\olm)$. The equivalence of (a) and $N_{\sub g(\oz)}(\olm)$ being a subspace is a well-known fact from convex analysis. 
\end{proof}

The following equivalent description of the critical cone of a polyhedral function was established in \cite[Proposition~3.2]{msa18}.
\begin{Proposition} [critical cone of polyhedral functions] \label{prop:cricone} 
Assume that $g:\R^m\to \oR$ is a polyhedral function and $(\oz, \olm)\in \gph \partial g$. If $\olm $ has the decomposition in  \eqref{decomp}, then  $w\in \Kg$ if and only if 
$w$ satisfies the conditions 
\begin{equation*}\label{cricone3}
\begin{cases}
\la a^i-a^j, w\ra = 0,\quad&\textrm{for }\; i,j \in J_+(\oz, \osigma),\\
\la a^i-a^j, w\ra \leq 0 ,&\textrm{for }\; i \in J(\oz)\setminus J_+(\oz, \osigma)\; \textrm{ and }\; j \in J_+(\oz, \osigma),\\
\la b^i, w\ra = 0,&\textrm{for }\; i \in I_+(\oz, \otau),\\
\la b^i, w\ra\leq 0,&\textrm{for }\; i \in I(\oz)\setminus I_+(\oz, \otau).
\end{cases}
\end{equation*}
\end{Proposition}

\section{Reduction Lemma for Polyhedral Functions and its Applications}\label{sec03}

We begin this section by providing an extension of the reduction lemma for polyhedral functions. The reduction lemma,
   established first by Robinson in \cite[Proposition~4.4]{rob} for polyhedral convex sets,    shows 
 that the graph of the normal cone to a polyhedral convex set coincides locally with that of the normal cone to its critical cone.
 Robinson's proof of this result relies upon his sticky face lemma, which was established in \cite[Lemma~3.5]{rob2} with a rather involved proof. 
 A simpler proof of this result  was presented by Dontchev and Rockafellar in \cite[Lemma~2E.4]{DoR14}. Recently, the reduction lemma was 
 extended for an important class of convex functions,  called piecewise linear-quadratic (cf. \cite[Definition~10.20]{rw}),  by the third author in \cite[Theorem~2.3]{s20}.
 This class of convex functions clearly     encompasses polyhedral functions and thus the result below can be derived  from the recent result in \cite{s20}.
 However, the presented proof in \cite{s20} relies upon two major results: 1) Robinson's reduction lemma for polyhedral convex sets and 2) the fact that the graph 
 of subgradient mappings of convex piecewise linear-quadratic functions can be expressed as a finite union of polyhedral convex sets (see the proof of \cite[Theorem~11.14(b)]{rw}).
 Below, we present a direct proof of   the reduction lemma for polyhedral functions, which is solely based on Lemma~\ref{lem:indset}.

\begin{Theorem}[reduction lemma]\label{thm:reduction}
Assume that $g:\R^m\to \oR$ is a polyhedral function and $(\oz, \olm)\in \gph \partial g$.  Then there exists $r>0$ such that
\begin{equation}\label{reduction}
\big( (\gph \sub g ) - (\oz, \olm)\big)\cap \B_r(0,0) = \big(\gph N_\Kg\big)\cap \B_r(0,0).
\end{equation}
\end{Theorem}

\begin{proof}
Suppose that the polyhedral function $g$  has the representation   \eqref{polyfunc}. We first show that there is $r>0$ for which
the inclusion `$\supset$' in \eqref{reduction} holds. To do so, using  \eqref{tanrep} for the polyhedral convex set $\partial g(\oz)$, we find $r_0>0$ such that 
\begin{equation}\label{relem1}
T_{\partial g(\oz)}(\olm)\cap \B_{r_0}(0) = \big(\partial g(\oz) -\olm\big)\cap \B_{r_0}(0).
\end{equation}
Pick now a pair $(w, u)$ from the right-hand side of \eqref{reduction} with $r=r_0$ therein, which implies that  the conditions $w\in \Kg\cap\B_{r_0}(0)$, $u \in \Kg^*\cap \B_{r_0}(0)$, and $\la u, w\ra =0$
are satisfied. Moreover, using \eqref{cricone1} for the polyhedral function $g$, we get $\Kg^*= T_{\partial g(\oz)}(\olm)$. Therefore $u$ belongs to the left-hand side of \eqref{relem1}, and thus $\olm+u \in\partial g(\oz)$. We now show that $(\oz +w, \olm +u) \in\gph\partial g$. To this end, take $z\in \R^m$   and observe that 
\begin{equation}\label{relem2}
\la \olm+u, z- (\oz+w)\ra = \la \olm+u, z-\oz\ra -\la\olm, w\ra - \la u, w\ra\leq g(z)-g(\oz) - \d g(\oz)(w),
\end{equation}
where the last inequality is deduced from the facts that $\olm +u \in \partial g(\oz)$, $w\in \Kg$, and $\la u, w\ra =0$. Choosing a smaller radius $r_0$ if necessary, we can assume that  \eqref{approx} is valid on
 $T_{\dom g}(\oz)\cap \B_{r_0}(0)$. It follows from Proposition~\ref{fopg}(a) and $w\in \Kg\cap \B_{r_0}(0)$ that $w\in T_{\dom g}(\oz)\cap \B_{r_0}(0)$, which in turn allows us to conclude via \eqref{approx} that
  $g(\oz)+\d g(\oz)(w) = g(\oz+w)$. Combining this and  \eqref{relem2} leads us to 
\begin{equation*}
\la \olm+u, z- (\oz+w)\ra \leq g(z)-g(\oz+w)
\end{equation*}
for any arbitrary $z\in \R^m$, and thus  to $\olm+u \in \partial g(\oz+w)$. This proves that  $(w, u)$ belongs to the left-hand side of \eqref{reduction} with $r=r_0$.

We now proceed with proving the inclusion `$\subset$' in \eqref{reduction} for some $r>0$. Decompose $\olm$ into $\olm = \olm_1+\olm_2$ with $\olm_1, \olm_2$ taken from  \eqref{decomp}. 
Pick $r_0$ from \eqref{relem1} and choose $r\in (0, r_0]$  such that the inclusions in \eqref{indset7} hold for all $(z, v)\in \big(\gph\partial g\big)\cap\B_{r}(\oz, \olm)$. Pick now a pair $(w, v)$ from the left-hand side of \eqref{reduction}, and conclude that  
 $(w, u) \in \B_{r}(0, 0)$ with $\olm+u \in \partial g(\oz+w)$. We are going to show that  $u\in N_\Kg(w)$, which amounts via \cite[Proposition~2A.3]{DoR14} to the conditions
\begin{equation}\label{relem4}
w\in \Kg, \quad u\in \Kg^*, \quad\textrm{ and }\quad \la u, w\ra =0.
\end{equation}
Taking a smaller radius $r$ if necessary, we can assume that $\partial g(\oz+w)\subset \partial g(\oz)$ (cf. \cite[Proposition~3.3(a)]{MoS16}). Thus, $u \in \big(\partial g(\oz) - \olm\big)\cap \B_{r}(0)$. Since $r\leq r_0$, we deduce from  \eqref{relem1} that 
 $u\in T_{\partial g(\oz)}(\olm)= \Kg^*$, which justifies the second inclusion  in \eqref{relem4}. By shrinking $r$ again if necessary, we can derive  from  $(\oz +w, \olm+u) \in  \big(\gph\partial g\big)\cap\B_{r}(\oz, \olm)$ and \eqref{indset7} that 
\begin{equation*}
J_+(\oz, \osigma) \subset J(\oz+w)\subset J(\oz) \quad\textrm{ and }\quad I_+(\oz, \otau)\subset I(\oz+w)\subset I(\oz).
\end{equation*}
 To prove that $w\in \Kg$, we use  Proposition~\ref{prop:cricone} in which an equivalent description of $\Kg$ was given. We break this task into four cases as follows:
\begin{itemize}[noitemsep,topsep=2pt]
\item [ \rm {(i)}]   $i, j \in J_+(\oz, \osigma)$. In this case, we obtain from  the inclusions above  that $i, j \in  J(\oz+w)\subset J(\oz)$. These inclusions bring us via \eqref{indset4} to  
\begin{equation*}
\la a^i, \oz\ra -\alpha_i = \la a^j, \oz\ra -\alpha_j\quad\textrm{ and }\quad
\la a^i, \oz+w\ra -\alpha_i = \la a^j, \oz+w\ra -\alpha_j.
\end{equation*}
Combining  these confirms that $\la a^i-a^j, w\ra  =0$ for all $i, j \in J_+(\oz, \osigma)$.

\item [ (ii)]  $i\in J(\oz)\setminus J_+(\oz, \osigma)$ and $j\in J_+(\oz, \osigma)$. In this case, we arrive at $i, j\in J(\oz)$ and $j\in J(\oz+v)$, which imply via \eqref{indset4} that  
\begin{equation*}
\la a^i, \oz\ra -\alpha_i = \la a^j, \oz\ra -\alpha_j\quad\textrm{ and }\quad
\la a^i, \oz+w\ra -\alpha_i \leq \la a^j, \oz+w\ra -\alpha_j.
\end{equation*}
Combining  these confirms that $\la a^i-a^j, w\ra  \le 0$ for all $i\in J(\oz)\setminus J_+(\oz, \osigma)$ and $j\in J_+(\oz, \osigma)$.

\item [{(iii)}]  $i\in I_+(\oz, \otau)$. In this case, we have $i\in I(\oz)$ and $i \in I(\oz+w)$, which result in $\la b^i, \oz\ra =\la b^i, \oz+w\ra=\beta_i$. Combining  these confirms that $\la b^i, w\ra=0$ for all $i\in I_+(\oz, \otau)$.

\item[(iv)]  $i\in I(\oz)\setminus I_+(\oz, \otau)$. In this case,  we deduce from $\oz+w\in \dom g$ and $i\in I(\oz)$ that $\la b^i, \oz+w\ra \leq \beta_i=\la b^i, \oz\ra$, which in turn yields  
$\la b^i, w\ra\le 0$ for all $i\in I(\oz)\setminus I_+(\oz, \otau)$.
\end{itemize}
In summary, we showed that $w$ satisfies the equivalent description of $\Kg$ from Proposition~\ref{prop:cricone}, and so $w\in \Kg$. 
It remains to demonstrate that $\la w,u\ra=0$. Recall that $\olm+u\in \partial g(\oz+w)$. Using \eqref{subg}, we can express $\olm+u$ as
\begin{equation*}\label{relem5}
\olm + u = \sum_{j\in J(\oz+w)} \sigma_j a^j+ \sum_{i\in I(\oz+w)} \tau_ib^i\; \; \textrm{ with }\; \; \sigma_j, \tau_i \geq 0\; \textrm{ and }\;  \sum_{j\in J(\oz+w)} \sigma_j =1.
\end{equation*}
This and the decomposition $\olm = \olm_1+\olm_2$ with $\olm_1,\olm_2$ taken from \eqref{decomp} allow us to conclude that  
\begin{align*}
\la u, w\ra &= \la \olm+u, w\ra-\la \olm, w\ra\\
&= \sum_{j\in J(\oz+w)} (\sigma_j-\osigma_j)\la a^j, w\ra+ \sum_{i\in I(\oz+w)} (\tau_i-\otau_i)\la b^i, w\ra\\
&= \sum_{j\in J(\oz+w)} (\sigma_j-\osigma_j)\big[\la a^j, \oz+w\ra-\alpha_j -(\la a^j, \oz\ra-\alpha_j)\big]+ \sum_{i\in I(\oz+w)} (\tau_i-\otau_i)\big[\la b^i, \oz+w\ra - \la b^i, \oz\ra\big]\\
&= \big(g(\oz+w)-g(\oz)\big)\sum_{j\in J(\oz+w)} (\sigma_j-\osigma_j)+ \sum_{i\in I(\oz+w)} (\tau_i-\otau_i)(\beta_i-\beta_i)\\
&=0,
\end{align*}
where  the second equality  results from the inclusions $J_+(\oz, \osigma)\subset J(\oz+w)$ and  $ I_+(\oz, \otau)\subset I(\oz+w)$, 
 where   the fourth equality comes from the  definitions of the active index sets $J(\oz+w)$ and $I(\oz+w)$ as well as the inclusions $J(\oz+w)\subset J(\oz)$ and $I(\oz+w)\subset I(\oz)$, 
 and where  the last equality follows from the assumption on $\sigma_j$ and $\osigma_j$. This proves the inclusion `$\subset$' in \eqref{reduction} for some $r>0$ and thus completes the proof. 
\end{proof}

The rest of this section will be devoted to presenting several important consequences of Theorem~\ref{thm:reduction} for different second-order variational constructions of polyhedral functions. 
We begin with a duality relation between critical cones of a polyhedral function and   its (Fenchel) conjugate in the sense of convex analysis.
Both results in the following corollary were observed by Rockafellar in \cite[equations~(3.23) \& (3.29)]{r22}. Since an explicit proof was not presented for \eqref{cricone2}
in \cite{r22} and since this result plays an important role in strict twice epi-differentiability of polyhedral functions in the next section, we supply a short proof for readers' convenience.
Let us recall an important result, used in the proof of next corollary, saying that the conjugate function $g^*$ of  a polyhedral function $g$ is again a polyhedral function, cf. \cite[Theorem~11.14(a)]{rw}. 

\begin{Corollary}[polar relation of critical cones] \label{polarity}
Assume that $g:\R^m\to \oR$ is a polyhedral function and $(\oz, \olm)\in \gph \partial g$.  Then $\Kg$ enjoys the duality relationship 
\begin{equation}\label{cricone2}
\Kg =  K_{g^*}(\olm, \oz)^*.
\end{equation}
Consequently, we have 
\begin{equation}\label{ricond}
\olm \in \ri \partial g(\oz) \; \textrm{ if and only if } \; \oz \in \ri \partial g^*(\olm).
\end{equation}
\end{Corollary}

\begin{proof}
Let $r>0$ be such that  \eqref{reduction} holds. Then, we conclude from \eqref{cricone1} that 
\begin{align*}
w \in \Kg \; &\Longleftrightarrow\; (w, 0)\in\gph N_\Kg\\
&\Longleftrightarrow\; \exists\, r'>0: t(w, 0)\in\big(\gph N_\Kg\big)\cap \B_r(0,0)\;\; \textrm{ for all } \; t\in [0, r')\\
&\Longleftrightarrow\; \exists\, r'>0: (\oz + t w, \olm)\in \big(\gph \partial g\big)\cap \B_r(\oz, \olm)\;\; \textrm{ for all } \; t\in [0, r')\\
&\Longleftrightarrow\; \exists\, r'>0: \olm \in\partial g(\oz+tw)\;\; \textrm{ for all } \; t\in [0, r')\\
&\Longleftrightarrow\; \exists\, r'>0: \oz+tw \in\partial g^*(\olm)\;\; \textrm{ for all } \; t\in [0, r')\\
&\Longleftrightarrow\; w\in T_{\partial g^*(\olm)}(\oz)=\big(N_{\sub g^*(\olm)}(\oz)\big)^*=K_{g^*}(\olm, \oz)^*,
\end{align*}
where the third equivalence follows from \eqref{reduction} and the last  {equivalence relies on the polyhedrality of $\partial g^*(\olm)$}.
To justify  \eqref{ricond}, it follows from Proposition~\ref{nonde1} that $\Kg$ is a linear subspace if and only if $\olm\in \ri \sub g(\oz)$.
A similar observation can be made for $K_{g^*}(\olm, \oz)$, which together with \eqref{cricone2} completes the proof of  \eqref{ricond}.
\end{proof}

We proceed to characterize critical cones of  polyhedral functions for  
points nearby a given point in the graph of their subgradient mappings. Such a result for polyhedral convex sets can be found in \cite[Lemma~4H.2]{DoR14}.
Recall that   a closed face $F$ of a  polyhedral convex cone $C\subset \R^d$ is defined by 
$$
F=C\cap [v]^\perp\quad \mbox{for some}\;\; v\in C^*.
$$
\begin{Proposition} \label{lem:faces}
Assume that $g:\R^m\to \oR$ is a polyhedral function and $(\oz, \olm)\in \gph \partial g$.
Then there exists $r>0$ such that for all $(z, \lm)\in \big(\gph \partial g\big)\cap \B_r(\oz, \olm)$ the corresponding critical cone $\Kgzv$ enjoys the following representation
\begin{equation*}
\Kgzv = F_1 -F_2,
\end{equation*}
for some faces $F_1, F_2$ of $\Kg$ with $F_2\subset F_1$. Conversely, for any pair of faces $F_1, F_2$ of $\Kg$ with $F_2\subset F_1$ and any sufficiently small real number $r>0$, there exists $(z, \lm)\in \big(\gph \partial g\big)\cap \B_r(\oz, \olm)$ with $\Kgzv = F_1-F_2$.
\end{Proposition}

\begin{proof} 
Let $r>0$ be such that   \eqref{reduction} holds. Choosing a smaller radius $r$ if necessary, we can assume 
without loss of generality that the inclusions $\sub g(z)\subset \sub g(\oz)$ and $N_{\partial g(\oz)}(\lm)\subset N_{\partial g(\oz)}(\olm)$ hold for any $(z, \lm) \in \big(\gph \partial g\big)\cap \B_r(\oz, \olm)$, which can be guaranteed by    $g$ being a polyhedral function. 
We first show that for any  $(z, \lm) \in \big(\gph \partial g\big)\cap \B_r(\oz, \olm)$, we have 
\begin{equation}\label{co6}
\Kgzv = \Kg\cap [\lm-\olm]^\perp +[z-\oz ].
\end{equation}
To this end, pick $(z, \lm) \in \big(\gph \partial g\big)\cap \B_r(\oz, \olm)$ and observe that the inclusion 
\begin{equation}\label{in1}
N_{\partial g(z)}(\lm) \subset N_{\partial g(z)}(\olm)
\end{equation}
holds. To justify it, it is not hard to see $\sub g(z)=\sub g(\oz)\cap D$ with $D:=\{v\in \R^m|\, \la v,z-\oz\ra=g(z)-g(\oz)\}$. 
It follows from $\lm\in \sub g(z)\subset D$ and $\lm-\olm\in N_\Kg(z-\oz)$ that $\olm \in D$. This, together with the definition of $D$, leads us to $N_D(\lm)=N_D(\olm)=[z-\oz]$.
Since both $\sub g(\oz)$ and $D$ are 
polyhedral, we conclude from the intersection rule for normal cones, holding without any constraint qualification for polyhedral convex sets, that 
\begin{eqnarray*}
N_{\partial g(z)}(\lm)&=& N_{\sub g(\oz)\cap D}(\lm)= N_{\partial g(\oz)}(\lm)+N_D(\lm)\\
&\subset& N_{\partial g(\oz)}(\olm) +N_D(\olm) =N_{\sub g(\oz)\cap D}(\olm)=N_{\partial g(z)}(\olm),
\end{eqnarray*}
proving our claimed inclusion.  Since $g^*$ is a polyhedral function as well,  using Lemma~\ref{lem:indset} for $g^*$ at $(\olm,\oz)\in \gph \sub g^*$ and shrinking $r$, if necessary, ensure that $\oz\in \sub g^*(\lm)$.
Moreover, it results from \eqref{in1} that $N_{\partial g^*(\lm)}(z) \subset N_{\partial g^*(\lm)}(\oz)$, which  in combination with \eqref{cricone1}  
leads us to 
\begin{equation*}\label{co5}
K_{g^*}(\lm, z)=N_{\partial g^*(\lm)}(z) = N_{\partial g^*(\lm)}(\oz)\cap [z-\oz]^\perp =K_{g^*}(\lm, \oz)\cap [z-\oz]^\perp.
\end{equation*}
Using this and  \eqref{cricone2} brings us to 
\begin{equation}\label{co8}
\Kgzv =K_{g^*}(\lm, z)^*= \big(K_{g^*}(\lm, \oz)\cap [z-\oz]^\perp\big)^*=K_g(\oz, \lm) +[z-\oz ].
\end{equation}
Similarly, by  \eqref{cricone1}   and the definition of the normal cone to the polyhedral convex set $\sub g(\oz)$, we obtain 
\begin{equation*}\label{co7}
K_g(\oz, \lm) = N_{\partial g(\oz)}(\lm) =N_{\partial g(\oz)}(\olm)\cap [\lm-\olm]^\perp=  \Kg\cap [\lm-\olm]^\perp,
\end{equation*}
which, together with \eqref{co8}, proves \eqref{co6}.    

After these preparations, we are in a position to justify the claimed descriptions of critical cones of $g$. Pick $(z, \lm) \in \big(\gph \partial g\big)\cap \B_r(\oz, \olm)$ and
set $F_1:= \Kg\cap [\lm-\olm]^\perp$, which is clearly a face of $\Kg$. Because $\lm-\olm\in N_\Kg(z-\oz)$, resulting from \eqref{reduction}, we conclude that $z-\oz\in F_1$. 
Since the relative interiors of  nonempty  faces of $F_1$ form a partition of this set (cf. \cite[Theorem~18.2]{Roc70}), we find a face of $F_1$, denoted by $F_2$, that $z-\oz\in \ri F_2$.
This tells us that  $F_2\subset F_1$. Moreover,   $F_2$ is a  face  of $\Kg$ as well.
By \eqref{co6}, the inclusion $\Kgzv \subset F_1 - F_2$ clearly holds. To get the opposite inclusion, pick $x_i\in F_i$ for $i=1,2$. It follows from $z-\oz\in \ri F_2$ and \cite[Theorem~6.4]{Roc70} that there is $t>0$ such that $(1+t)(z-\oz)-tx_2\in F_2\subset F_1$,
which yields 
\begin{equation}\label{in2}
t(x_1-x_2)= tx_1+(1+t)(z-\oz)-tx_2- (1+t)(z-\oz)\in F_1+[z-\oz],
\end{equation}
which confirms via \eqref{co6} that  the inclusion $F_1-F_2\subset \Kgzv$ holds. This shows that  $ \Kgzv=F_1-F_2$.

Assume now that $F_1$ and $F_2$ are faces of $\Kg$ with $F_2\subset F_1$. 
Pick $r>0$ from Theorem \ref{thm:reduction} and choose $u\in \Kg^*$ with $\|u\|<r/2$ such that $F_1 = \Kg\cap [u]^\perp$. Since $u\in \Kg^*=N_\Kg(0)$, it follows from Theorem \ref{thm:reduction} that $\olm+u \in \partial g(\oz)$.
Pick now $w\in \ri F_2$ with $\|w\|<r/2$ and  observe from $w\in F_1$ that $(w, u) \in \gph N_\Kg$.  By Theorem \ref{thm:reduction}, we arrive at $(\oz+w, \olm+u) \in \gph \partial g$. Since $w\in \ri F_2$, 
as in \eqref{in2}, we can show that $F_1-F_2=F_1 +[w]$.  Employing this and  \eqref{co6} leads us to  
\begin{equation*}
F_1 - F_2 = F_1 +[w] =  \Kg\cap [u]^\perp +[w] = K_g(\oz+w, \olm+u),
\end{equation*}
which completes the proof.
\end{proof}

\begin{Corollary}\label{incri}
For a polyhedral function $g:\R^m\to \oR$  and $(\oz, \olm)\in \gph \partial g$, the following properties hold.
\begin{itemize}[noitemsep,topsep=2pt]
\item [ \rm {(a)}] There exists $r>0$ such that for all $(z, \lm) \in \big(\gph\partial g\big)\cap \B_r(\oz, \olm)$ the following inclusions hold:
\begin{subequations}\label{cri2}
\begin{align}
&\Kgzv \subset \Kg -\Kg,\label{cri2:a}\\
&\Kg \cap -\Kg \subset \Kgzv.\label{cri2:b}
\end{align}
\end{subequations}
\item [ \rm {(b)}] If,  in addition,     $\olm \in \ri\partial g (\oz)$, then there exists $r>0$ such that for all $(z, \lm)\in \big(\gph \partial  g\big)\cap\B_r(\oz, \olm)$, we have $\lm\in \ri \sub g(z)$ and 
\begin{equation}\label{cri1}
\Kgzv = \Kg.
\end{equation}
\end{itemize}
\end{Corollary}
\begin{proof} The inclusion in 
 \eqref{cri2:a} for any $(z, \lm) \in \gph \partial g$ sufficiently close to $(\oz, \olm)$ falls directly from Proposition~\ref{lem:faces}. 
Also, according to Proposition~\ref{lem:faces}, for any such $(z, \lm) \in \gph \partial g$, one can find  some faces $F_1, F_2$ of $\Kg$ with $F_2\subset F_1$
such that $K_g(z,\lm)=F_1-F_2$. Since $\Kg \cap -\Kg$ is the smallest face of $\Kg$, we get $\Kg \cap -\Kg\subset F_i$ for $i=1,2$, which in turn justify 
the inclusion in  \eqref{cri2:b} for all $(z, \lm) \in \gph \partial g$ sufficiently close to $(\oz, \olm)$. 

To justify (b), it follows from $\olm\in \ri\partial g(\oz)$ and  Proposition~\ref{nonde1}  that $\Kg$ is a linear subspace.  This and the inclusions  \eqref{cri2:a} and \eqref{cri2:b}   tell us  that
\begin{equation*}
\Kg=\Kg-\Kg\supset \Kgzv \supset \Kg\cap-\Kg=\Kg
\end{equation*}
for all $(z, \lm)\in \gph \partial  g$ sufficiently close to $(\oz,\olm)$, which in turn confirms that  $\Kgzv$ is a linear subspace. Using Proposition~\ref{nonde1} again, we have $\lm\in \ri \sub g(z)$ and hence complete the proof of (b).
\end{proof}

 Both inclusions in Corollary~\ref{incri}(a) were perviously established in \cite[Proposition~2E.10]{DoR14} for polyhedral convex sets without appealing to the 
reduction lemma,   an important tool used in our proof. Note that while the observation in Corollary~\ref{incri}(b) is a simple and direct consequence of the inclusions  \eqref{cri2:a}-\eqref{cri2:b},
it plays an indispensable role in the next two sections in which we are going to study strict twice epi-differentiability of polyhedral functions. In fact, this result reveals that when we 
are converging in the graph of subgradient mappings of polyhedral functions to a given point therein under the extra  relative interior condition,  all those points, used in this convergence, enjoy this relative interior condition. 
This has a major implication for   second-order variational constructions such as limiting coderivatives and strict subgradient graphical derivatives; see Theorem~\ref{lem:proto1}(c) and Theorem~\ref{limnc}(c).



We proceed with two other important applications of the established  reduction lemma for polyhedral functions in calculating the  proto-derivative and coderivative of subgradient mappings of this class of functions.
To this end, consider a set-valued mapping $F:\R^n\tto \R^m$. According to \cite[Definition~8.33]{rw}, the {\em graphical derivative} of $F$ at $\ox$ for $\oy$ with $(\ox,\oy) \in \gph F$ is the set-valued mapping $DF(\ox, \oy): \R^n\tto \R^m$ defined via the tangent cone to $\gph F$ at $(\ox, \oy)$ by
\begin{equation*}
u \in DF(\ox, \oy)(w)\; \Longleftrightarrow\; (w,u)\in T_{\gph F}(\ox, \oy),
\end{equation*}
or, equivalently, $\gph DF(\ox, \oy) = T_{\gph F}(\ox, \oy)$.  Using the definition of the tangent cone, we can present an alternative definition of $DF(\ox, \oy)$ in terms of graphical limits as
\begin{equation}\label{proto}
\gph DF(\ox, \oy) = \limsup_{t\searrow 0}\frac{\gph F-(\ox,\oy)}{t}.
\end{equation}
The set-valued mapping $F$ is said to be {\em proto-differentiable} at $\ox$ for $\oy$ if the outer graphical limit in \eqref{proto} is actually a full limit. If $F$ is proto-differentiable at $\ox$ for $\oy$,   its graphical derivative $DF(\ox, \oy)$ is   called the 
{\em proto-derivative} of $F$ at $\ox$ for $\oy$. When $F(\ox)$ is a singleton consisting of $\oy$ only, the notation $DF(\ox, \oy)$ is simplified to $DF(\ox)$. It is easy to see that 
for a single-valued function $F$, which is differentiable at $\ox$,  the graphical derivative  $DF(\ox)$ boils down to the Jacobian matrix of $F$ at $\ox$, denoted by $\nabla F(\ox)$. 
Recall from \cite[Definition~9.53]{rw} that the strict graphical derivative 
of  a set-valued mapping $F$ at $\ox$ for $\oy$ with $(\ox,\oy)\in \gph F$, is  the set-valued mapping $D_*F(\ox,\oy):\R^n\tto \R^m$, defined  by 
\begin{equation}\label{sproto}
\gph D_*F(\ox,\oy)=\limsup_{
  \substack{
   t\searrow 0 \\
  (x,y)\xrightarrow{ \gph F}(\ox,\oy)
  }} \frac{\gph F-(x,y)}{t}.
\end{equation}
The set-valued mapping $F$ is said to be {\em strictly} proto-differentiable at $\ox$ for $\oy$ if the outer graphical limit in \eqref{sproto} is   a full limit.
 When $F$ is strictly proto-differentiable at $\ox$ for $\oy$, its strict graphical derivative $D_*F(\ox, \oy)$ is   called the {\em strict} proto-derivative of $F$ at $\ox$ for $\oy$. 
 To elaborate  more,  take a set  $\Omega\subset \R^d$ and  $\ox\in \Omega$. The regular (Clarke) tangent cone and the paratingent cone to   $\Omega$ at $\ox$ are defined, respectively,   by
\begin{equation}\label{tancc}
 \widehat T_{\Omega}(\ox) =\liminf_{x \xrightarrow{ \Omega}\ox,t\searrow 0} \frac{\Omega-x}{t}\quad \mbox{and}\quad  \widetilde {T}_{\Omega}(\ox) =\limsup_{x \xrightarrow{ \Omega}\ox,t\searrow 0} \frac{\Omega-x}{t}.
\end{equation}
The strict proto-differentiability of $F$ at $\ox$ for $\oy$ amounts to saying that $ \widehat T_{\gph F}(\ox,\oy)= \widetilde {T}_{\gph F}(\ox,\oy)$.
Observe also that since the inclusions $ \widehat T_{\gph F}(\ox,\oy)\subset  T_{\gph F}(\ox,\oy)\subset  \widetilde {T}_{\gph F}(\ox,\oy)$ always hold, 
strict proto-differentiability of $F$ at $\ox$ for $\oy$ implies that  its proto-derivative and strict proto-derivative coincide, namely we have 
\begin{equation}\label{eqspd}
D_*F(\ox,\oy)=DF(\ox,\oy).
\end{equation}
 
Strict proto-differentiability of a set-valued mapping has  far-reaching consequences. In fact, it allows us to evaluate the strict proto-derivative 
of a set-valued mapping, which can be utilized to characterize its strong metric regularity; see \cite[Theorem~4D.1]{DoR14} and Section~\ref{sec05} for more details on this 
application. The intriguing question is whether  strict proto-differentiability holds for any class of set-valued mappings,  important for constrained and composite optimization problems.  In the next theorem, we are going to show that subgradient mappings of polyhedral functions enjoy this property when a relative interior condition is satisfied for   subgradients under consideration. Furthermore, we will 
prove that such a relative interior condition, indeed, characterizes strict proto-differentiability of these set-valued mappings. While it is not easy to study strict proto-differentiability 
of set-valued mappings, its weaker version, namely proto-differentiability, has been well understood for many important classes of functions that are  important for various applications; 
see \cite{r89,mms, ms20, rw} for more details and examples. In particular, it is well-known (cf. \cite[Corollary~13.41]{rw} ) that subgradient mappings of polyhedral functions are always 
proto-differentiable.

\begin{Theorem}\label{lem:proto1}
For a polyhedral function $g:\R^m\to \oR$  and $(\oz, \olm)\in \gph \partial g$, the following properties hold.
\begin{itemize}[noitemsep,topsep=2pt]
\item [ \rm {(a)}]  The graphical derivative of  $  \sub g$  at $\oz$ for $\olm$ can be obtained by  
\begin{equation}\label{proto4}
\gph D(\sub  g)(\oz,\olm) = T_{\gph\partial g}(\oz, \olm)= \gph N_\Kg.
\end{equation}

\item [ \rm {(b)}]  The strict graphical derivative of  $  \sub g$  at $\oz$ for $\olm$ can be obtained by{\footnote{This formula was brought to our attention by a reviewer.}}  
\begin{equation}\label{proto4n}
\gph D_*(\sub  g)(\oz,\olm) = \gph N_{\Kg}-\gph N_{\Kg}.
\end{equation}

\item [ \rm {(c)}]   The condition   $\olm \in \ri\partial g (\oz)$ amounts to the existence of  a neighborhood $O$ of $(\oz,\olm)$ such that for any $(z,\lm)\in O\cap \gph \sub g$, $\sub g$ is strictly proto-differentiable at $z$ for $\lm$
and its proto-derivative and strict proto-derivative  coincide, namely 
\begin{equation}\label{gdpd}
D_*(\sub  g)(z,\lm)=D(\sub  g)(z,\lm).
\end{equation}
\end{itemize}
\end{Theorem}

\begin{proof} Observe that the first equality in \eqref{proto4} is the definition of graphical derivative. Using the equality \eqref{reduction} and the definition of the tangent cone to $\gph \sub g$ at $(\oz, \olm)$, we immediately  arrive at the second equality in \eqref{proto4}.

To justify (b),  we conclude from Theorem~\ref{thm:reduction} and \cite[Proposition~4H.10]{DoR14}, respectively, that 
\begin{equation}\label{para}
 \widetilde T_{\gph \partial g}(\oz,\olm)=\gph D_*(\sub  g)(\oz,\olm)= \gph D_*N_{\Kg}(0,0)=\gph N_{\Kg} - \gph N_{\Kg}.
\end{equation}

Finally, we proceed with the proof of (c). Suppose that  $\olm\in \ri \sub g(\oz)$.  By Corollary~\ref{incri}(b), we find $r>0$ such that for any $(z,\lm)\in \B_r(\oz,\olm)\cap \gph \sub g$, the property 
 $K_g(z,\lm)=\Kg$ holds. Take any such a pair $(z,\lm) $ and conclude from part (a) that  $$T_{\gph \partial g}(z,\lm)=  \gph N_{K_g(z,\lm)}=\gph N_\Kg.$$
 It follows then from \cite[Theorem~6.26]{rw} that 
 $$
\widehat T_{\gph \partial g}(\oz,\olm)=\liminf_{
  (z,\lm)\xrightarrow{   \gph \sub g}(\oz,\olm)}
  T_{\gph \partial g}(z,\lm)= \gph N_\Kg.
 $$
Since $\olm\in \ri \sub g(\oz)$, we conclude that $\gph N_{\Kg}=\Kg\times\Kg^\perp$ is a linear subspace. These, coupled with \eqref{para}, lead us to 
 $\widehat T_{\gph \partial g}(\oz,\olm)= \widetilde T_{\gph \partial g}(\oz,\olm)$, meaning that $\sub g$ is strictly proto-differentiable at $\oz$ for $\olm$.
 To achieve a similar conclusion for any pair $(z,\lm)\in \gph \sub g$ sufficiently close to $(\oz,\olm)$, observe from Corollary~\ref{incri}(b)
that for any such a pair, we have $\lm\in \ri \sub g(z)$. A similar argument as the one presented above for $(\oz,\olm)$, shows that $\sub g$ is strictly   proto-differentiable at $z$ for $\lm$
whenever $(z,\lm)\in \gph \sub g$  is sufficiently close to $(\oz,\olm)$. Finally,  \eqref{gdpd} results from \eqref{eqspd}. Conversely, if $\sub g$ is strictly proto-differentiable at $\oz$ for $\olm$, 
our discussion prior to \eqref{eqspd} implies that
$$
\gph N_{\Kg}=T_{\gph \partial g}(\oz,\olm)= \widetilde T_{\gph \partial g}(\oz,\olm)=\gph N_{\Kg}-\gph N_{\Kg},
$$
where the last equality results from \eqref{para}. This clearly tells us that  $\Kg$ is a linear subspace and hence it follows from Proposition~\ref{nonde1} that  $\olm\in \ri \sub g(\oz)$, which completes the proof.
\end{proof}

We continue with one more direct application of the reduction lemma for polyhedral functions in finding the regular and limiting normal cones to   $\gph \sub g$.

\begin{Theorem}[regularity of subgradient mappings of polyhedral functions] \label{limnc}
For a polyhedral function $g:\R^m\to \oR$  and $(\oz, \olm)\in \gph \partial g$, the following properties hold.
\begin{itemize}[noitemsep,topsep=2pt]
\item [ \rm {(a)}]   The regular   normal cone to  $\gph \sub g$ at $(\oz,\olm)$ can be calculated by   
\begin{equation}\label{co1}
\rN_{\gph \partial g}(\oz, \olm) = \Kg^*\times \Kg.
\end{equation}
\item [ \rm {(b)}] The {\rm(}limiting{\rm)} normal cone to   $\gph \sub g$ at $(\oz,\olm)$ can be calculated by 
\begin{equation}\label{co4}
N_{\gph \partial g}(\oz, \olm) = \bigcup_{\substack{
   F_1, F_2\subset  \mathcal{F}_\Kg\\
  F_2\subset F_1
  }} (F_1-F_2)^*\times (F_1-F_2),
\end{equation}
where $\mathcal{F}_\Kg$ is the collection of all faces of $\Kg$.
\item [ \rm {(c)}]  The condition   $\olm\in \ri \sub g(\oz)$  is equivalent to 
 \begin{equation}\label{co2}
N_{\gph \partial g}(\oz, \olm)=\rN_{\gph \partial g}(\oz, \olm).
\end{equation}
\end{itemize}
\end{Theorem}

\begin{proof}
We begin with the proof of (a). It follows from Theorem~\ref{thm:reduction} that $\gph \sub g- (\oz,\olm)$ and $\gph N_\Kg$ coincide locally around $(0,0)$, which 
leads us via the definition of the regular normal cone to 
$$
\rN_{\gph \partial g}(\oz,\olm)= \rN_{\gph \partial g -(\oz,\olm)}(0,0)=\rN_{\gph N_{\Kg} }(0,0).
$$
Since $\Kg$ is a polyhedral convex set, it results from  \cite[page~264]{DoR14} that 
$$
\rN_{\gph N_{\Kg} }(0,0)= K_{\Kg}(0,0)^*\times K_{\Kg}(0,0)= \Kg^*\times \Kg.
$$
Combining these proves (a). 

To justify (b), observe that  $(u,w)\in N_{\gph \partial g}(\oz, \olm)$ if and only if there are  sequences $(z^k,\lm^k)\to (\oz,\olm)$ with $(z^k,\lm^k) \in \gph \sub g$ and 
$(u^k,w^k)\to (u,w)$ with $(u^k,w^k)\in \rN_{\gph \partial g}(z^k,\lm^k)$.
Using the definition of the regular normal cone and \eqref{reduction}, we conclude  for any $k$ sufficiently large that 
\begin{eqnarray*}
(u^k,w^k)\in \rN_{\gph \partial g}(z^k,\lm^k)  &\iff &(u^k,w^k)\in \rN_{\gph \partial g -(\oz,\olm)}(z^k-\oz,\lm^k-\olm) \\
&\iff &(u^k,w^k)\in \rN_{\gph N_{\Kg} }(z^k-\oz,\lm^k-\olm).
\end{eqnarray*}
By the definition of the limiting normal cone, we conclude from the last inclusion that $(u,w)\in N_{\gph N_{\Kg} }(0,0)$ and thus get  the inclusion $N_{\gph \partial g}(\oz, \olm)\subset N_{\gph N_{\Kg} }(0,0)$. A similar argument via \eqref{reduction} can be used 
to justify the opposite inclusion and obtain 
$$
N_{\gph \partial g}(\oz, \olm)= N_{\gph N_{\Kg} }(0,0).
$$
Recall that  $\Kg$ is a polyhedral convex cone. Employing the representation of the limiting normal cone to  the normal cone 
to the graph of a polyhedral convex set, obtained in the proof of \cite[Theorem~2]{dr96}, we arrive at the representation  on the right-hand side of \eqref{co4}  for $N_{\gph N_{\Kg} }(0,0)$, which 
completes the proof of (b). 

To prove (c), observe that it follows from $\olm\in \ri \sub g(\oz)$ and Corollary~\ref{incri}(b) that there exists a neighborhood $O$ of $(\oz,\olm)$ such that for any $(z,\lm)\in O\cap \gph \sub g$, we have 
$K_g(z,\lm)=\Kg$. This and   \eqref{co1} bring us to 
 \begin{equation}\label{co22}
\rN_{\gph \partial g}(z, \lm)=\rN_{\gph \partial g}(\oz, \olm)\quad \mbox{for all}\;\; (z,\lm)\in O\cap \gph \sub g, 
\end{equation}
which, together with  the definition of the limiting normal cone, justifies the   claimed equality in \eqref{co2}.  Conversely, suppose that 
\eqref{co2} holds. We are going to show that $\Kg$ is a linear subspace. To this end, it suffices to show that if $w\in \Kg$, then $-w\in \Kg$. Pick $w\in \Kg$ and conclude from \eqref{proto4} that 
there is $u\in \R^m$ such that $(w,u)\in T_{\gph \partial g}(\oz, \olm)$. Appealing now to \cite[Theorem~13.57]{rw} implies that $(u,-w)\in N_{\gph \partial g}(\oz, \olm)$. This, together with \eqref{co2} and \eqref{co1},
tells us that $-w\in \Kg$. Remember that $\Kg=N_{\sub g(\oz)}(\olm)$. Since $\Kg$ is a linear subspace, we arrive at  $\olm\in \ri \sub g(\oz)$, which completes the proof of (c).
\end{proof}
The description \eqref{co1} of the regular normal cone to $\gph \sub g$ in terms of the critical cone of $g$ was established in \cite[Theorem~4.3(i)]{MoS16} using a different approach.
Our current proof relies upon the reduction lemma, which  allows us to   simplify the proof of this result. We should mention that a similar result
was established for polyhedral convex sets using Robinson's  reduction lemma for polyhedral convex sets in the proof of \cite[Theorem~2]{dr96}.
A similar expression of the limiting normal cone to $\gph \sub g$ was achieved in \cite[Theorem~5.1]{MoS16} via a lengthy direct argument.
 Our proof, which  heavily uses Theorem~\ref{thm:reduction}, reduces the calculation to  the case of   a polyhedral convex cone   and then utilizes 
the available result for this setting. Thus, Theorem~\ref{limnc} can be considered as a generalization of   Dontchev and Rockafellar's result, obtained in \cite[Theorem~2]{dr96}, for polyhedral convex sets.
Note that part (c) of Theorem~\ref{limnc} offers a new piece of information about $\gph \sub g$, which has not been observed before to the best of our knowledge.
Indeed, it tells us that $\gph \sub g$ is {\em regular} at $(\oz,\olm)\in \gph \sub g$ in the sense of \cite[Definition~6.4]{rw} if and only if  the subgradient $\olm$ is taken from the relative interior of $\sub g(\oz)$.

Recall that for a polyhedral function $g$ with $(\oz,\olm)\in \gph \sub g$,  the coderivative mapping of   $\sub g$ at $\oz$ for $\olm$, denoted  $D^*(\sub g) (\oz,\olm )$,   is defined by 
$$
u\in D^*(\sub g) (\oz,\olm )(w)\iff (u,-w)\in N_{\gph \partial g}(\oz, \olm).
$$
It is well-known (cf. \cite[Theorem~13.57]{rw}) that the inclusion $D(\sub g) (\oz,\olm )(w)\subset D^*(\sub g) (\oz,\olm )(w)$, for any $w\in \R^m$, always holds for a polyhedral function $g$. Below, we show that 
this inclusion becomes equality provided that the subgradient $\olm$ is taken from the relative interior of $\sub g(\oz)$.
\begin{Corollary}\label{codgh} Assume that $g:\R^m\to \oR$ is a polyhedral function and  that   $(\oz, \olm)\in \gph \partial g$. 
Then $\olm\in \ri \sub g(\oz)$ if and only if 
\begin{equation}\label{dce}
D(\sub g)(\oz,\olm)(w)=D^*(\sub g)(\oz,\olm)(w)\quad \mbox{ for all}\;\; w\in \R^m.
\end{equation}
\end{Corollary}
\begin{proof} Suppose that $\olm\in \ri \sub g(\oz)$. Thus  the critical cone  $\Kg$ is a linear subspace. By Theorem~\ref{lem:proto1}(a), 
we have $T_{\gph \partial g}(\oz, \olm)=\gph N_\Kg=\Kg\times \Kg^\perp$. Moreover, we conclude from Theorem~\ref{limnc}(c)
that $N_{\gph \partial g}(\oz, \olm)=\Kg^\perp\times \Kg $. Combining these with the definitions of coderivative and proto-derivative  justifies \eqref{dce}.

Assume now that \eqref{dce} holds. We are going to show that $\Kg$ is a linear subspace. To this end, it suffices to show that if $w\in \Kg$, then $-w\in \Kg$. Pick $w\in \Kg$ and conclude from \eqref{proto4} that 
there is $u\in \R^m$ such that $u\in D(\partial g)(\oz, \olm)(w)=N_\Kg(w)$, which yields that $(w,u)\in\Kg\times \Kg^*$. This, combined with  \eqref{co4}, tells us that $(u,w)\in \Kg^*\times \Kg\subset N_{\gph \partial g}(\oz, \olm)$.
Employing  the definition of coderivative and \eqref{dce}, we obtain $u \in D^*(\sub g)(\oz,\olm)(-w)=D(\sub g)(\oz,\olm)(-w)=N_\Kg(-w)$, which clearly confirms that $-w\in \Kg$. 
Since $\Kg=N_{\sub g(\oz)}(\olm)$ and $\Kg$ is a linear subspace, we obtain $\olm\in \ri \sub g(\oz)$, which  completes the proof.
\end{proof}

\section{Strict Twice Epi-Differentiability of Polyhedral Functions}\label{sec4}
In this section, we study another important second-order variational property, called strict twice epi-differentiability, for  polyhedral functions and  continuous differentiability of 
proximal mappings for this class of functions. Once again, our results rely heavily on our extension of the reduction lemma for polyhedral functions. 
To achieve our goal, consider  a function $f: \R^n \to \oR$ with $\ox \in \R^n$ with $f(\ox)$ finite and  define the parametric  family of 
second-order difference quotients of $f$ at $\ox$ for $\ov\in \sub f(\ox)$ by 
$$
\Delta_t^2 f(\bar x , \ov)(w)=\frac{f(\ox+tw)-f(\ox)-t\langle \ov,\,w\rangle}{\frac {1}{2}t^2}
$$
for any $w\in \R^n$ and $t>0$. The {second subderivative} of $f$ at $\ox$ for $\ov$, denoted $\d^2 f(\bar x , \ov)$, is an extended-real-valued function defined  by 
\begin{equation*}\label{ssd}
\d^2 f(\bar x , \ov)(w)= \liminf_{\substack{
   t\searrow 0 \\
  w'\to w
  }} \Delta_t^2 f(\ox , \ov)(w'),\;\; w\in \R^n.
\end{equation*}
Following \cite[Definition~13.6]{rw},   $f$ is said to be {twice epi-differentiable} at $\bar x$ for $\ov$
if  the functions $  \Delta_t^2 f(\bar x , \ov)$ epi-converge to $  \d^2 f(\bar x,\ov)$ as $t\searrow 0$.  Further, we say that $f$ is {\em strictly} twice epi-differentiable at $\bar x$ for $\ov$ if 
the functions $  \Delta_t^2 f(  x , v)$ epi-converge  to a function as $t\searrow 0$, $(x,v)\to (\ox,\ov)$ with $f(x)\to f(\ox)$ and $(x,v)\in \gph \sub f$.  If this condition holds, the limit function is then the second subderivative $  \d^2 f(\bar x,\ov)$.
Twice epi-differentiability of extended-real-valued functions, introduced by Rockafellar in \cite{r85},  has been investigated for important classes of functions appearing in constrained and composite optimization 
problems   in \cite{r88, mms,ms20}. Its strict version, introduced in \cite{pr}, was only  studied in  \cite{pr3} for nonlinear programming and minimax problems and so it is tempting to 
ask when this property holds    and to explore its applications in parametric optimization. To this end, we define the {\em strict second subderivative} of $f$ at 
$\bar x$ for $\ov$ with  $\ov \in \partial f(\ox)$ at $w\in \R^n$ by 
\begin{equation*}
\d_s^2f(\ox, \ov)(w) = \liminf_{\substack{
t\searrow 0, \, w'\to w \\
(x, v) \xrightarrow{\gph \partial f}(\ox, \ov)\\
f(x)\to f(\ox)}} 
\Delta_t^2 f(x , v)(w') .
\end{equation*}
When $f$ is subdifferentially continuous at $\ox$ for $\ov$ in the sense of \cite[Definition~13.28]{rw}, we can drop the requirement $f(x)\to f(\ox)$ in the definition of the strict second subderivative.
According to \cite[Example~13.30]{rw}, convex functions are always subdifferentialy continuous. 
Clearly, we always have $\d_s^2f(\ox, \ov)(w)\le \d^2f(\ox, \ov)(w)$ for any $w\in \R^n$. If $f$ is strictly  twice epi-differentiable at $\bar x$ for $\ov$, the latter inequality becomes equality. 
Being able to  calculate the strict second subderivative of a function and comparing with its second subderivative can tell us when strict twice epi-differentiability should be expected for such a function. 
Note also that the second subderivative was exploited to characterize the quadratic growth condition for extended-real-valued functions in \cite[Theorem~13.24(c)]{rw}. Similarly, 
we can use the strict second subderivative to achieve a characterization of the uniform quadratic growth condition (cf. \cite[Definition~5.16]{bs}), which plays an important role in parametric optimization. 
This is beyond the scope of this paper and so we postpone it to our future research. 
Below, we use the characterization of the faces of the critical cone 
of a polyhedral function in Proposition~\ref{lem:faces} to find its strict second subderivative.

\begin{Proposition}\label{ssdp}
For a polyhedral function $g: \R^m \to \oR$  and  $(\oz, \olm)\in \gph \partial g$, its strict second subderivative can be calculated by 
\begin{equation}\label{st}
\d_s^2 g(\oz, \olm)(w) = \delta_{\Kg-\Kg}(w) \quad\textrm{ for all }\; w\in \R^m.
\end{equation}
\end{Proposition}
\begin{proof}
We first prove that $\d_s^2 g(\oz, \olm)(w)=0$ for all $w\in\Kg-\Kg$. First, observe that 
the convexity of $g$ yields $\d_s^2 g(\oz, \olm)(w)\geq 0$ for all $w\in \R^m$. To obtain the opposite inequality, pick 
$w\in\Kg-\Kg$. By Proposition~\ref{lem:faces}, we find a sequence $(\zk, \lmk)\to(\oz, \olm)$ such that  $(z^k,\lm^k)\in \gph\partial g$ and  $K_g(\zk, \lmk)= \Kg-\Kg$. This  implies that $w\in K_g(\zk, \lmk)$, 
which is equivalent to saying that $\d g(z^k)(w)=\la \lmk,w\ra$. Employing Proposition~\ref{fopg}(b), we find a sequence $t_k\searrow 0$ such that $g(z^k+t_kw)-g(z^k)=\d g(z^k)(t_kw)=t_k \d g(z^k)(w)$.
By definition, we get $\Delta_{t_k}^2 g(\zk , \lmk)(w)=0$, which results in 
\begin{equation*}
\d_s^2 g(\oz, \olm)(w) \le \lim_{k\to \infty} \Delta_{t_k}^2 g(\zk , \lmk)(w)=0.
\end{equation*}
This confirms that  $\d_s^2 g(\oz, \olm)(w)= 0$ for any $w\in\Kg-\Kg$. To obtain \eqref{st}, it suffices to 
 to justify that $\dom \d_s^2 g(\oz, \olm) \subset \Kg-\Kg$. Taking any $w\in \R^m$ with $\d_s^2 g(\oz, \olm)(w)<\infty$, we can find sequences $t_k\searrow 0$, $\wk\to w$, and $(\zk, \lmk)\to (\oz, \olm)$ with $(\zk, \lmk)\in \gph\partial g$ such that 
\begin{equation*}
\lim_{k\to\infty}\frac{g(\zk+t_k\wk)-g(\zk)-t_k\langle \lmk, \wk\rangle}{\tfrac{1}{2}t_k^2} <\infty.
\end{equation*}
This, together with the convexity of $g$, allows us to find a constant $M>0$ such that 
\begin{equation*}
0\leq\frac{g(\zk+t_k\wk)-g(\zk)-t_k\langle \lmk, \wk\rangle}{t_k}\leq Mt_k 
\end{equation*}
for all $k$ sufficiently large and therefore to obtain 
\begin{equation}\label{st2}
\lim_{k\to\infty}\frac{g(\zk+t_k\wk)-g(\zk)-t_k\langle \lmk, \wk\rangle}{t_k}=0.
\end{equation}
Take an arbitrary 
$$
u\in \big(\Kg-\Kg\big)^\perp= \Kg^*\cap -\Kg^*= T_{\partial g(\oz)}(\olm)\cap -T_{\partial g(\oz)}(\olm),
$$
where the last equality results from \eqref{cricone1}.  We now show that $\langle u, w\rangle =0$, which in turn yields  $w\in \Kg - \Kg$. To this end, by \eqref{tanrep}, 
we find   $\alpha >0$ such that $\olm\pm\alpha u \in \partial g(\oz)$. This clearly implies  that
\begin{equation}\label{st1}
\langle \olm\pm\alpha u, \zk+t_k\wk-\oz\rangle\leq g(\zk+t_k\wk)-g(\oz).
\end{equation}
Employing  Theorem~\ref{thm:reduction}, we can conclude that  $(\zk - \oz, \lmk - \olm) \in \gph N_\Kg$ for all $k$ sufficiently  large. This tells us that   $\zk-\oz\in \Kg=N_{\partial g(\oz)}(\olm)$ for all $k$ sufficiently large. Moerover,  Lemma~\ref{lem:indset} 
confirms that  $\olm \in \partial g(\zk)$ for all $k$ sufficiently  large. Combining these, we can conclude that 
\begin{equation*}
\langle u, \zk-\oz\rangle = 0\quad\textrm{ and }\quad \langle \olm, \zk - \oz\rangle = g(\zk)-g(\oz).
\end{equation*}
These, together with \eqref{st1}, yield
\begin{equation*}
\langle \olm\pm\alpha u, t_k\wk\rangle\leq g(\zk +t_k\wk)-g(\zk),
\end{equation*}
and therefore we get 
\begin{equation*}
\langle \olm - \lmk \pm\alpha u, \wk\rangle \leq \frac{g(\zk+t_k\wk)-g(\zk)-t_k\langle \lmk, \wk\rangle}{t_k}.
\end{equation*}
Passing to the limit as $k\to\infty$ and employing \eqref{st2}, we arrive at  $\pm\alpha\langle u, w\rangle \leq 0$, meaning  $\langle u, w\rangle = 0$. This confirms that  $w\in \Kg - \Kg$ and hence completes the proof.
\end{proof}

The established formula for the strict second subderivative of a polyhedral function in \eqref{st} suggests   a path forward in the study of strict twice epi-differentiability of this class of functions.
As pointed out earlier, strict twice epi-differentiability requires that the second subderivative and strict second subderivative coincide. For polyhedral functions, Proposition~\ref{ssdp}
immediately suggests that the given subgradient $\olm$ must belong to  $\ri \sub g(\oz)$; see the proof of the implication (a)$\implies$(b) in Theorem~\ref{riste} for a detailed proof.
One may wonder whether the opposite holds as well, namely the relative interior condition $\olm\in \ri \sub g(\oz)$ implies strict twice epi-differentiability of polyhedral functions.
Our next goal is to indeed demonstrate that this is true.  In doing so, we rely heavily upon  the reduction lemma, obtained in Theorem~\ref{thm:reduction}, as well as 
a characterization of strict twice epi-differentiablity of prox-regular functions from \cite{pr2}. 
The next result is a simplified version of  \cite[Corollary~4.3]{pr2},  and presents a useful characterization of strict twice epi-differentiability of convex functions, 
which  comprise an important subclass of prox-regular functions
according to \cite[Example~3.30]{rw}.

\begin{Proposition}[characterization of strict twice epi-differentiability] \label{sted}
Assume that $f:\R^n\to \oR$, $\ox\in \R^n$ with $f(\ox)$ finite, and $\ov\in \sub f(\ox)$  and that $f$ is a convex function.
Then there is a neighborhood  $O$ of $(\ox,\ov)$ such that for any $(x,v)\in O\cap \gph f$, the following properties are equivalent:
\begin{itemize}[noitemsep,topsep=2pt]
\item [ \rm {(a)}] $f$ is strictly twice epi-differentiable at $x$ for $v$;
\item [ \rm {(b)}] $\sub f$ is strictly proto-differentiable at $x$ for $v$.
\end{itemize}
\end{Proposition}

The next result is an immediate consequence of Theorem~\ref{lem:proto1}(c) and reveals that polyhedral functions are always strictly twice epi-differentiable  
  under a relative interior condition. 
\begin{Theorem}[strict twice epi-differentiability of polyhedral functions]\label{riste}
Assume that $g:\R^m\to \oR$ is a polyhedral function and  that $(\oz,\olm)\in \gph \sub g$. Then 
 the following  properties are equivalent:
\begin{itemize}[noitemsep,topsep=2pt]
\item [ \rm {(a)}]  there is a neighborhood $O$ of $(\oz,\olm)$ such that for any $(z,\lm)\in O\cap \gph \sub g$, $g$ is strictly twice epi-differentiable at $z$ for $\lm$;
\item [ \rm {(b)}] $\olm \in \ri \partial g(\oz)$.
\end{itemize}
\end{Theorem}

\begin{proof} We begin with the implication  (b)$\implies$(a). By Proposition~\ref{sted}, (a) is equivalent to strict proto-differentiability of $\sub g$ at  $z$ for $\lm$ for any $(z,\lm)\in O\cap \gph \sub g$.
It follows from Theorem~\ref{lem:proto1}(c) that the latter condition holds since $\olm \in \ri \partial g(\oz)$. This shows that (b) yields (a).  
 To prove the opposite implication, suppose that (a) holds. 
As pointed out earlier, when strict twice epidifferentiablity holds for a function, its second subderivative and   strict  second subderivative coincide. 
Strict twice epi-differentiability of  $g$ at $\oz$ for $\olm$, together with \eqref{st},  tells us that 
$$
\dd_{K_g(\oz,\olm)-\Kg}= \d_s^2 g(\oz,\olm)= \d^2 g(\oz,\olm)=\dd_{K_g(\oz,\olm)},
$$
where the last equality comes from \cite[Propsoition~13.9]{rw}. 
 This implies that $K_g(\oz,\olm)-\Kg=\Kg$, meaning that $\Kg$ is a linear subspace. By Proposition~\ref{nonde1}, we get $\olm\in \ri \sub g(\oz)$, which completes the proof.
\end{proof}

Combining the obtained characterization of strict twice epi-differentiability  of polyhedral functions with \eqref{ricond} allows us to 
conclude that  this property is preserved   under the Fenchel conjugate for polyhedral functions, as shown below. 

\begin{Corollary} Assume that $g:\R^m\to \oR$ is a polyhedral function and  that $(\oz,\olm)\in \gph \sub g$. Then 
the following properties are equivalent:
\begin{itemize}[noitemsep,topsep=2pt]
\item [ \rm {(a)}]  there is a neighborhood $O$ of $(\oz,\olm)$ such that for any $(z,\lm)\in O\cap \gph \sub g$, $g$ is strictly twice epi-differentiable at $z$ for $\lm$;
\item [ \rm {(b)}]  there is a neighborhood $U$ of $(\olm,\oz)$ such that for any $(\lm,z)\in U\cap \gph \sub g^*$, $  g^*$ is strictly twice epi-differentiable at $\lm$ for $z$;
\item [ \rm {(c)}]  $\olm\in \ri \sub g(\oz)$;
\item [ \rm {(d)}]  $\oz\in \ri \sub g^*(\olm)$.
\end{itemize}
\end{Corollary} 

\begin{proof}
We know from \cite[Theorem~11.14(a)]{rw} that $g^*$ is a polyhedral function. We obtain the equivalence of   (a) and (c) and of  (b) and (d) from  Theorem~\ref{riste}.
Corollary~\ref{polarity} also tells us that (c) and (d) are equivalent, which completes the proof.
\end{proof}

We close this section with another   consequence of the reduction lemma for polyhedral function 
about smoothness of the Moreau envelope and  proximal mapping for polyhedral functions. 
To this end,  recall  that for a function $f\colon\R^n\to\oR$ and parameter value  $r>0$,  the {  Moreau envelope} function  $e_r f$ and the proximal mapping $\prox_{rf}$ are defined, respectively,  by 
\begin{equation*}\label{moreau}
e_r f(x)=\inf_{w\in \R^n}\Big\{f(w)+\frac{1}{2r}\|w-x\|^2\Big\},
\end{equation*}
and
\begin{equation*}\label{proxmap}
\prox_{r f}(x)= \argmin_{w\in \R^n}\Big\{f(w)+\frac{1}{2r}\|w-x\|^2\Big\}.
\end{equation*}
 When $f$ is convex,  the subdifferential sum rule from convex analysis 
   implies that 
\begin{equation}\label{prsub}
{\rm{prox}}_{rf}(x)=\big(I+r\sub f\big)^{-1}(x),\;\;x\in \R^n,
\end{equation}
where $I$ stands for the $n\times n$ identity matrix. Furthermore, it is known that the envelope function $e_rf$ is continuously differentiable (cf. \cite[Theorem~2.26]{rw}) for any convex function $f$.
If, in addition, $f$ is a polyhedral function, we deduce from \cite[Proposition~13.9]{rw} and \cite[Exercise~13.45]{rw} that  the proximal mapping $\prox_{rf}$ is  semidifferentiable.
According to \cite[Proposition~2D.1]{DoR14}, the latter is equivalent to directional differentiability of  $\prox_{rf}$. 
Below, we present a simple but useful characterization of continuous differentiability of the proximal mapping of polyhedral functions.

\begin{Theorem}\label{prodiff} Assume that $g:\R^m\to \oR$ is a polyhedral function and  that $(\oz,\olm)\in \gph \sub g$. Then 
the following properties are equivalent:
\begin{itemize}[noitemsep,topsep=2pt]
\item [ \rm {(a)}]  $\olm\in \ri \sub g(\oz)$;
\item [ \rm {(b)}]  for any $r>0$, the envelope function $e_r g$ is  ${\cal C}^2$ in a neighborhood of  $\oz+r\olm$;
\item [ \rm {(c)}]  for any $r>0$, the proximal mapping $\prox_{rg}$ is ${\cal C}^1$ in a neighborhood of  $\oz+r\olm$.
\end{itemize}
Furthermore, if  $\olm\in \ri \sub g(\oz)$ and $r>0$, then for any $x$ sufficiently close to $\oz+r\olm$, the Jacobian matrix $\nabla (\prox_{rg})(x)$  and the Hessian matrix $\nabla^2(e_{r}g)(x)$ can be calculated, respectively, by 
$$
\nabla (\prox_{rg})(x)=P_{K_g(\oz,\olm)}\quad \mbox{and}\quad \nabla^2(e_{r}g)(x)=\frac{1}{r}\big(I-P_{K_g(\oz,\olm)}\big)=\frac{1}{r}P_{K_{g^*}(\olm,\oz)}.
$$
\end{Theorem}

\begin{proof} Set $\ph(z)=g(z)-\la \olm,z\ra$ for any $z\in \R^m$. Since $\olm\in \sub g(\oz)$, we get $0\in \sub \ph(\oz)$. This, together with the convexity of $\ph$, yields 
$\oz\in \mbox{argmin}\,  \ph$, where $\mbox{argmin}\,  \ph$ stands for the set of global minimizers of $\ph$ over $\R^m$. 
Observe also that the condition $\olm\in \ri \sub g(\oz)$ is equivalent to $0\in \ri \sub \ph (\oz)$
and that 
\begin{equation}\label{proxch}
\prox_{r\ph}(z)=\prox_{rg}(z+r\olm)\quad \mbox{for all}\;\; z\in \R^m.
\end{equation}
 Picking  $r>0$ and employing \cite[Theorem~4.4]{pr2} imply that $\prox_{r\ph}$ is ${\cal C}^1$ in a neighborhood of  $\oz$ if and only if $\ph$ is strictly twice epi-differentiable at $z$ for $v$   
for all $(z,v)\in \gph \sub \ph$   sufficiently close to $(\oz,0)$ (the parameter $r$ in \cite[Theorem~4.4]{pr2} should be chosen sufficiently small
since the function under consideration in \cite{pr2} is prox-regular. It is, however, well-known that such a restriction on $r$ for convex functions is not necessary).
By Theorem~\ref{riste}, the latter property of $\ph$ amounts to the condition $0\in \ri \sub \ph(\oz)$. 
It is not hard to see   that $\ph$ is strictly twice epi-differentiable at $z$ for $v$  
for  any $(z,v) \in \gph \sub \ph$   sufficiently close to $(\oz,0)$ if and only if $g$ enjoys the same property at $z$ for $\lm$ for any $(z,\lm)\in \sub g$  sufficiently close to $(\oz,\olm)$.
Combining these with \eqref{proxch} and Theorem~\ref{riste}, we conclude the equivalence of (a) and (c). To obtain the equivalence of (a) and (b), 
one can  see that 
$$
e_{r}\ph(z)=e_{r}g(z+r\olm)-\la \olm,z\ra-\frac{r}{2}\|\olm\|^2\quad \mbox{for all}\;\; z\in \R^m.
$$
This, combined with a similar argument via \cite[Theorem~4.4]{pr2}, confirms that (a) and (b) are equivalent. 

Finally, pick $r>0$ and suppose  that $\olm\in \ri \sub g(\oz)$. By (b) and (c), there exists a neighborhood  $U$ of $\oz+r\olm$ such 
that $\prox_{rg}$ is ${\cal C}^1$ and $e_r g$ is  ${\cal C}^2$ on $U$. 
 Take  $x\in U$   and set $y=\prox_{rg}(x)$.  
By \eqref{prsub}, we get $v:=r^{-1}(x-y)\in \sub g(y)$. It also follows from \eqref{prsub} and the definition of graphical derivative  that 
\begin{equation*}
D (\prox_{rg})(x)(w)=\big(I+rD(\sub g)(y,v)\big)^{-1}(w)\quad \mbox{for all}\;\; w\in \R^m,
\end{equation*}
which, together with \eqref{proto4}, brings us to 
$$
D (\prox_{rg})(x)(w)= \big(I+r N_{K_g(y,v)}\big)^{-1}(w)=\big(I+N_{K_g(y,v)}\big)^{-1}(w)=P_{K_g(y,v)}(w)
$$
for any $w\in \R^m$. By (c), the proximal mapping $\prox_{rg}$ is differentiable at $x$ and thus we get $D (\prox_{rg})(x)=\nabla (\prox_{rg})(x)$.
Combining these confirms the claimed formula for the Jacobian matrix $\nabla (\prox_{rg})(x)$. Recall also from \cite[Theorem~2.26]{rw} that 
$\nabla(e_rg)(x)=r^{-1}(x-\prox_{rg}(x))$ for any $x\in \R^m$. By (b), the envelope function $e_rg$ is twice differentiable at $x$ and thus we have 
$$
  \nabla^2(e_{r}g)(x)=\frac{1}{r}(I-\nabla (\prox_{rg})(x)\big)=\frac{1}{r}\big(I-P_{K_g(y,v)}\big)=\frac{1}{r} P_{K_{g^*}(v,y)},
$$
where the last equality comes from the identity $P_{K_g(y,v)}+P_{K_g(y,v)^*}=I$ together with \eqref{cricone2}. 
Now, we claim that $K_g(y,v)=\Kg$ whenever $x\in U$. This can be accomplished via Theorem~\ref{incri}(b) provided that 
we show $(y,v)\in \gph \sub g$ is sufficiently close to $(\oz,\olm)$. Since the proximal mapping is nonexpansive and since $\prox_{rg}(\oz+r\olm)=\oz$, we get 
$$
\|y-\oz\|=\| \prox_{rg}(x) -\prox_{rg}(\oz+r\olm)\|\le \|x-\oz-r\olm\|.
$$
Moreover, we have 
$$
\|v-\olm\|=\| r^{-1}(x-y)-\olm\| \le r^{-1}\|x-\oz-r\olm\|+r^{-1}\| y-\oz\|\le 2r^{-1}\|x-\oz-r\olm\|.
$$
Using these estimates and shrinking $U$, if necessary,  confirm our claim and hence prove the claimed formulas for the Jacobian matrix $\nabla (\prox_{rg})(x)$  and the Hessian matrix $\nabla^2(e_{r}g)(x)$. 
\end{proof}

We should  mention that smoothness of projection mapping onto a closed convex set was first studied by Holmes in \cite{hol} in Hilbert spaces.   His main result, \cite[Theorem~2]{hol}, states 
that if $C\subset \R^d$ is a closed convex set, $x\in \R^d$,  the boundary of  $C$ is a ${\cal C}^2$ smooth manifold  (cf.  \cite[Example~6.8]{rw}) around $y=P_C(x)$, then the projection mapping $P_C$ is ${\cal C}^1$ in a neighborhood of the open normal ray 
$\{y+t(x-y)|\; t>0\}$. As pointed out by Hiriart-Urruty in \cite{hu}, when the projection point $y$ is a {\em corner point}, Holmes's result can not be utilized to study smoothness of the projection mapping
because the boundary of $C$ fails to be  a ${\cal C}^2$ smooth manifold around $y$. In contrast, Theorem~\ref{prodiff} goes beyond the projection mapping and provides a characterization of smoothness of the proximal mapping of a polyhedral function
via a verifiable condition. While our result is limited to polyhedral functions, our approach via second-order variational analysis opens a new door to study smoothness of   projection mappings of convex sets.
It is important to emphasize that our approach to characterize smoothness of proximal mappings
demonstrates that instead of expecting smoothness of the boundary of the convex set under consideration, we should look for  a second-order regularity condition, which seems to be the driving force 
for such a result.  
 \begin{Corollary} \label{prodiff3} Assume that $C\subset \R^m$ is a polyhedral convex set and $x\in \R^m$. Then $P_C$ is ${\cal C}^1$ in a neighborhood of  $x$ if and only if $x-z\in \ri N_C(z)$, where $z=P_C(x)$.
 \end{Corollary}
 \begin{proof} Applying  Theorem~\ref{prodiff} to the polyhedral function $g=\dd_C$ proves the claimed equivalence. 
 \end{proof}

Note that  a characterization of differentiability at a point, but not continuous differentiability {\em around} a point,    of the projection mapping $P_C$, $C$ being a polyhedral convex set, 
via the same relative interior condition as in Corollary~\ref{prodiff3}  can be found in \cite[Corollary~4.1.2]{fp}. Not only is our proof different from the one in \cite{fp}, also 
 Corollary~\ref{prodiff3} improves the latter result by showing  that   differentiability of the projection mapping at a given point can be   strengthened to the ${\cal C}^1$ property of this mapping in a neighborhood of that point.

\section{ Regularity Properties of Variational Systems}\label{sec05}
In this section, we aim to explore the relationship between important regularity properties of the solution mapping to 
the generalized equation \eqref{vs1}.  When $g=\dd_C$ with $C$ being a polyhedral convex set, the generalized 
equation \eqref{vs1} will be an example of  variational inequalities. In this case, the seminal paper \cite{dr96} revealed for the first time that strong metric regularity 
and metric regularity of the solution mapping to the canonical perturbation of \eqref{vs1} are equivalent; see below for the definitions of both concepts. 
We aim to present   a similar result for the generalized equation \eqref{vs1}, which is based upon   the reduction lemma for polyhedral functions.
Furthermore, we show that for  nondegenerate solutions to \eqref{vs1}, meaning solutions that satisfy the condition \eqref{mr1}, one can find conditions under which 
the solution mapping to the generalized equation \eqref{vs1} is continuously differentiable.

To  explore regularity properties of \eqref{vs1}, define the set-valued mapping $G:\R^m\tto\R^m$ by 
\begin{equation}\label{mr3}
G(x):= \psi(x)+\partial g(x), \quad x\in \R^m,
\end{equation}
and then consider the solution mapping $S: \R^m \tto \R^m$ to the canonical perturbation of the generalized equation \eqref{vs1}  by
\begin{equation}\label{mr33}
S(y):= G^{-1}(y)=\big\{x\in \R^m\, \big|\, y\in \psi(x)+\partial g(x)\big\}, \quad y\in \R^m.
\end{equation}

Recall that a set-valued mapping $F:\R^n \tto \R^m$ is called {\em metrically regular} at $\ox$ for $\oy\in F(\ox)$ if   there exist $\kappa \geq 0$  and  neighborhoods $U$ of $\ox$ and $V$ of $\oy$ such that the distance estimate
\begin{equation}\label{mr8}
\dist \big(x, F^{-1}(y)\big)\leq \kappa\, \dist\big(y, F(x)\big)
\end{equation}
holds for all $(x, y)\in U\times V$. 
The mapping is called {\em strongly metrically regular} at $\ox$ for $\oy$ if $F^{-1}$ admits a Lipschitz continuous single-valued localization around $\oy$ for $\ox$, which means that 
there exist neighborhoods $U$ of $\ox$ and $V$ of $\oy$ such that the mapping $y\mapsto F^{-1}(y)\cap U$ is single-valued and Lipschitz continuous on $V$. According to 
\cite[Proposition~3G.1]{DoR14}, strong metric regularity of $F$ at $\ox$ for $\oy$ amounts to 
  $F$ being metrically regular at $\ox$ for $\oy$ and its inverse $F^{-1}$ admitting a single-valued  localization around $\oy$ for $\ox$.

  \begin{Theorem}[equivalence between metric regularity and strong metric regularity]\label{thm:MR1}  
Assume that   $\ox$ is a   solution to the generalized equation \eqref{vs1}. 
Then  the mapping $G$, taken from \eqref{mr3}, is metrically regular at $\ox$ for $0$ if and only if it is strongly metrically regular    at $\ox$ for $0$.
\end{Theorem}

\begin{proof} Note that one can argue via \cite[Corollary~3F.5]{DoR14} that (strong) metric regularity of the mapping $G$ from \eqref{mr3} at $\ox$ for $0$ is equivalent to that of the mapping $x\mapsto \psi(\ox)+\nabla\psi(\ox)(x-\ox)+\sub g(x)$ 
at $\ox$ for $0$. Employing then Theorem~\ref{thm:reduction} tells us that (strong) metric regularity of the latter amounts to the same property of the mapping $\Phi(w):= \nabla \psi(\ox)w+N_{\K}(w)$ with $\K= K_g(\ox, -\psi(\ox))$ and $w\in \R^m$
at $0$ for $0$. Since $\Phi$ fits into the framework of \cite{dr96}, one can conclude via \cite[Theorem~3]{dr96} (or  \cite[Corollary~9.7]{io}) that metric regularity and strong metric regularity are equivalent for $\Phi$ at $0$ for $0$, which implies that 
these properties are equivalent for $G$ at $\ox$ for $0$ and hence completes the proof. 
\end{proof}

We proceed with an application of Theorem~\ref{thm:MR1} in studying  regularity properties of  the solution mapping to the KKT system of the composite minimization problem
\begin{equation}\label{cp}
\mini \varphi(x)+\big(g\circ\Phi\big)(x), \quad\textrm{subject to} \; \; x\in \R^n,
\end{equation}
where $\varphi: \R^n \to \R$ and $\Phi: \R^n \to \R^m$ are $\C^2$ functions and $g:\R^m\to \oR$ is a polyhedral function. The KKT system associated with the composite problem \eqref{cp} is given by
\begin{equation}\label{kkt}
0=\nabla_xL(x, \lm), \quad \lm \in \partial g(\Phi(x)),
\end{equation}
where $L(x, \lm):= \varphi(x) +\la \lm, \Phi(x)\ra$ with  $(x, \lm)\in \R^n\times \R^m$ is the Lagrangian of \eqref{cp}. 
A pair $(\ox, \olm)$ is called a {\em KKT point} of \eqref{cp} provided that it satisfies  the KKT system \eqref{kkt}. Define the mapping $\Psi:\R^n\times\R^m \tto\R^n\times \R^m$ by
\begin{equation}\label{mr10}
\Psi(x, \lm):= \begin{bmatrix}
\nabla_xL(x, \lm)\\-\Phi(x)
\end{bmatrix} +\begin{bmatrix}
0\\\partial g^*(\lm)
\end{bmatrix}
\end{equation}
and observe that $(\ox, \olm)$ is a KKT point if and only if $(0, 0)\in \Psi(\ox, \olm)$. We aim at finding conditions under which the solution mapping to the canonical perturbed of the KKT system \eqref{kkt}, defined by
\begin{equation*}
S_{KKT}(p, q):= \Psi^{-1}(p, q)= \big\{(x, \lm) \in \R^n\times \R^m\, \big|\, (p,q) \in \Psi(x, \lm)\big\},
\end{equation*}
has a Lipschitz continuous single-valued localization. As shown below, this can be distilled from Theorem~\ref{thm:MR1}. 

\begin{Theorem}\label{mrkkt}
Let $(\ox, \olm)$ be a KKT point of \eqref{cp}. Then the following properties are equivalent:
\begin{itemize}[noitemsep,topsep=2pt]
\item [ \rm {(a)}] the mapping $\Psi$ is metrically regular at $(\ox, \olm)$ for $(0,0)$;
\item [ \rm {(b)}] the mapping $\Psi$ is strongly metrically regular at $(\ox, \olm)$ for $(0,0)$;
\item [ \rm {(c)}] the solution mapping $S_{KKT}$ has a Lipschitz continuous single-valued localization around $(0,0)$ for $(\ox, \olm)$;
 \item [ \rm {(d)}] the implication 
 $$
 \begin{cases}
 \nabla^2_{xx}L(\ox,\olm)w+\nabla\Phi(\ox)^*w'=0,\\
w'\in D^*(\sub g)( \Phi(\ox),\olm)(\nabla\Phi(\ox)w)
 \end{cases}
 \implies (w,w')=(0,0)
 $$
 holds.
 \end{itemize}
\end{Theorem}

\begin{proof}Set 
\begin{equation*}
\psi(x, \lm): = \begin{bmatrix}
\nabla_xL(x, \lm)\\
-\Phi(x)
\end{bmatrix} \quad\textrm{ and }\quad \hat g(x, \lm) := g^*(\lm),\;\; (x,\lm)\in \R^n\times \R^m,
\end{equation*}
and observe that   the KKT system \eqref{kkt} can be written as the generalized equation 
\begin{equation}\label{gekkt}
(0,0)\in \Psi(x, \lm) = \psi(x, \lm) +\partial \hat g(x, \lm).
\end{equation}
 It follows from  \cite[Theorem~11.14(a)]{rw} that $g^*$ is a polyhedral function and so is $\hat g$.
 Also, we deduce from \cite[Proposition~10.5]{rw} that 
 $\partial \hat g (x, \lm) = \{0\}\times \partial g^*(\lm)$. 
 The equivalence of (a)-(c) comes directly from Theorem~\ref{thm:MR1}. Part (d) is 
 a direct consequence of the coderivative criterion from \cite[Theorem~9.40]{rw} for metric regularity of the mapping $\Psi$   at $(\ox, \olm)$ for $(0,0)$, 
 which translates as the implication 
 $$
 \begin{bmatrix}0\\0\end{bmatrix} \in D^*\Psi \big((\ox,\olm),  (0,0)\big)(w,w') =\left[\begin{array}{c}
\nabla_{xx}^2 L(\ox,\olm)w -\nabla\Phi(\ox)^*w'\\
\nabla\Phi(\ox)w + D^*(\sub g^*)(\olm, \Phi(\ox))(w')\\
\end{array}
\right]\implies  (w,w')=(0,0).
 $$
The second inclusion in this implication amounts to $-w'\in D^*(\sub g)( \Phi(\ox),\olm)(\nabla\Phi(\ox)w)$. 
 Combining this with the above implication  confirms that (d) and (a) are equivalent and hence completes the proof.
\end{proof}

For   classical nonlinear programming problems (NLPs), it is well-known that metric regularity and strong metric regularity of KKT systems are equivalent; see \cite[Theorem~4I.2]{DoR14} and \cite[Section~7.5]{kk}.
 Theorem~\ref{mrkkt} extends this result for the composite problem \eqref{cp}. 

  According to Theorem~\ref{thm:MR1}, if $\ox$ is a solution to \eqref{vs1} and if  the mapping $G$ in \eqref{mr3} is metrically regular at $\ox$ for $0$, then the solution mapping $S$ in \eqref{mr33}
  has a  Lipschitz continuous single-valued localization around $0$ for $\ox$. 
  Our final  goal in  this section is to show that if, in addition, $\ox$ is a nondegenerate solution to \eqref{vs1} which means it satisfies \eqref{mr1}, then 
  the aforementioned localization of $S$ is continuously differentiable around $0$. To this end, we begin with
    showing that strict proto-differentiability is preserved for the sum of two functions.

\begin{Proposition}[sum rule for strict proto-derivative]\label{prop:sum} Let $f:\R^n \to \R^m$ be ${\cal C}^1$ around $\ox$ and let $F:\R^n\tto\R^m$ with $\oy \in F(\ox)$. Then we have
\begin{equation*}
D_*(f+F)(\ox, f(\ox)+\oy)(w) = \nabla f(\ox)w + D_*F(\ox, \oy)(w)\quad \mbox{for all}\;\;w\in \R^n.
\end{equation*}
Moreover,   $F$ is strictly proto-differentiable at $\ox$ for $\oy$ if and only if $f+F$ is strictly proto-differentiable at $\ox$ for $f(\ox)+\oy$.
\end{Proposition}

\begin{proof} The given formula for $D_*(f+F)(\ox, f(\ox)+\oy)$ was already established in \cite[Exercise~10.43(b)]{rw}. To justify the second claim, 
suppose that $F$ is strictly proto-differentiable at $\ox$ for $\oy$. Let   $(w, u)\in \gph D_*(f+F)(\ox, \oy+f(\ox))$ and take arbitrary sequences  $t_k\searrow 0$, and $(\xk, \zk)\to(\ox, \oy+f(\ox))$ with $\{(x^k,z^k)\b\subset \gph (f+F)$.  
The latter tells us that $(\xk, \yk)\to(\ox, \oy )$ with $y^k:=\zk-f(\xk)$. 
By the sum rule for the strict graphical derivative, we get $(w,u- \nabla f(\ox)w)\in \gph D_*F(\ox, \oy)$. Since $F$   is strictly proto-differentiable at $\ox$ for $\oy$, we find a sequence $(\wk, v^k) \to (w, u-\nabla f(\ox)w)$
such that $\yk+t_k v^k\in F(\xk+t_k\wk)$ for all $k\in \N$, which in turn implies for any $k\in \N$ that 
$$
 \zk+t_k u^k\in (f+F)(\xk+t_k\wk)\quad \mbox{with}\;\;u^k:=\frac{f(\xk+t_k\wk)-f(\xk)}{t_k}+ v^k.
$$
Since $\{(w^k,u^k)\b$ converges to $(w,u)$, we conclude that $f+F$   is strictly proto-differentiable at $\ox$ for $f(\ox)+\oy$.

Assume now that $f+F$ is strictly proto-differentiable at $\ox$ for $f(\ox)+\oy$. By the argument above, we can conclude that $F=f+F+(-f)$ is strictly proto-differentiable at $\ox$ for $\oy$,
which completes the proof. 
\end{proof}

Recall from \cite[page~173]{r85} that a set $C\subset \R^d$ is called {\em smooth} at $\ox\in C$ if the tangent cone $T_C(\ox)$ is a linear subspace of $\R^d$ and 
 the ``$\limsup$" in \eqref{tan1} is the ``$\lim$." It is called {\em strictly smooth} at $\ox$ if 
 $
 \widehat T_{C}(\ox) = \widetilde T_{C}(\ox),
 $
 where both tangent cones were defined in \eqref{tancc}.
 It follows from  \cite[Proposition~3.1]{r85} that if $f:\R^n\to \R^m$ is Lipschitz continuous around $\ox\in \R^n$, then $\gph f$ is (strictly) smooth at $(\ox,f(\ox))$ if and only if $f$ is (strictly) differentiable at $\ox$.  
\begin{Theorem}\label{slcl}
Assume that   $\ox$ is a nondegenerate solution to the generalized equation \eqref{vs1}. 
Then the solution mapping $S$ to    \eqref{vs1} has a Lipschitz continuous single-valued localization $\sigma$ around $0\in \R^m$ for $\ox$ if and only if one of the equivalent conditions
\begin{itemize}[noitemsep,topsep=2pt]
\item [ \rm {(a)}]   $\big(\nabla\psi(\ox)\K\big) +\K^\perp = \R^m$;
\item [ \rm {(b)}]   $\{w\in \R^m|\; \nabla \psi(\ox)^*w\in \K^{\perp}\big\}\cap \K=\{0\}$,
\end{itemize}
holds, where $\K=K_g(\ox, -\psi(\ox))$. 
 In this case, the function $\sigma$ is ${\cal C}^1$ in a neighborhood of  $0$ and 
$$
\nabla \sigma(y)=B\big(B^*\nabla \psi(\sigma(y))B\big)^{-1}B^*
$$
for all $y$ sufficiently close to $0$, where $B\in \R^{m\times s}$ is a matrix whose columns form a basis for the linear subspace $\K$ with $s=\dim \K$.
\end{Theorem}
\begin{proof} The equivalence of the conditions (a) and (b) falls directly from  \cite[Corollary~11.25(c)]{rw}. Part (b) is also a translation of the coderivative criterion from \cite[Theorem~9.40]{rw}
for metric regularity of the mapping $G$ in \eqref{mr3} at $\ox$ for $0$. Indeed, we have  for any $w\in \R^m$ that 
\begin{equation*}
D^*G(\ox,0)(w)=\nabla \psi(\ox)^*w+D^*(\sub g)(\ox,-\psi(\ox))(w)=  \nabla \psi(\ox)^*w+ N_{\K}(w),
\end{equation*}
where the last equality results from Theorems~\ref{codgh} and \ref{lem:proto1}(a). This reads as $u\in D^*G(\ox,0)(w)$ if and only if 
$u\in \nabla \psi(\ox)^*w+  {\K^\bot}$ and $w\in \K$. Using this, one can see that (b) amounts to saying that $G$ is metrically regular at $\ox$ for $0$.
By Theorem~\ref{thm:MR1}, the latter is equivalent to strong metric regularity of $G$ at $\ox$ for $0$, a property that amounts to saying that 
 the solution mapping  $S$ in \eqref{mr33} has a  Lipschitz continuous single-valued localization  around $0$ for $\ox$. 
 {So we find neighborhoods $U$ of $\ox$ and $V$ of $0$ such that the mapping $y\mapsto S(y)\cap U$ is single-valued and Lipschitz continuous on $V$. Define the function $\sigma:V\to U$
by $\sigma(y)=S(y)\cap U$ for any $y\in V$.} We are going to demonstrate that $\sigma$ is ${\cal C}^1$ in a neighborhood of  $0$. To this end,  
it follows from the nondegeneracy condition \eqref{mr1} and Theorem~\ref{lem:proto1}(c) that $\sub g$ is strictly proto-differentiable at $x$ for $z$ whenever $(x,z)\in \gph \sub g$ is sufficiently close to $(\ox,-\psi(\ox))$,
which, together with Proposition~\ref{prop:sum}, 
tells us that $G$ is strictly proto-differentiable at $x$ for $y$ whenever $(x,y)\in \gph G$ is sufficiently close to $(\ox,0)$. 
Suppose without loss of generality that $G$ is strictly proto-differentiable at $x$ for $y$ whenever $(x,y)\in (U\times V)\cap \gph G$.
Choose a pair $(x,y)\in (U\times V)\cap \gph G$ and observe that $\sigma$ is strictly proto-differentiable at $y$ for $x$. This
implies that  $\gph \sigma$ is strictly smooth at $(y,x)$. Since  $\sigma$ is Lipschitz continuous on $V$,  it follows from  \cite[Proposition~3.1]{r85} 
that  $\sigma$   is strictly differentiable at $y$. This means that $\sigma$ is strictly differentiable on $V$, a property   equivalent to saying that $\sigma$ is ${\cal C}^1$ on $V$ (cf. \cite[Exercise~1D.8]{DoR14}).

Finally, to justify the claimed formula for the Jacobian matrix of $\sigma$, take $y\in V$. Thus, for any $u\in \R^n$ we have 
$w=\nabla \sigma(y)u=D\sigma(y)(u)$, which is equivalent to $u\in DG(x,y)(w)$, where $x=\sigma(y)$. By the definition of $G$ from \eqref{mr3}, the sum rule for the graphical derivative from \cite[Exercise~10.43(b)]{rw}, and Theorem~\ref{lem:proto1}(a), we 
obtain for any $(x,y)\in \gph G$ and $w\in \R^m$ that 
\begin{eqnarray}
DG(x,y)(w)=\nabla \psi(x)w+D(\sub g)(x,y-\psi(x))(w)=\nabla \psi(x)w+N_{K_g(x,y-\psi(x))}(w).\nonumber
\end{eqnarray}
Recall that  $\ox$ is a nondegenerate solution to \eqref{vs1}. Shrinking the neighborhoods $U$ and $V$, if necessary, we deduce from Corollary~\ref{incri}(b) that for any $(x,y)\in  (U\times V)\cap \gph G$, 
we have $K_g(x,y-\psi(x))=K_g(\ox, -\psi(\ox))=\K$. Combining these  with $\K$ being a linear subspace brings us to 
\begin{equation}\label{gder}
u\in DG(x,y)(w)\iff u\in \nabla \psi(x)w+\K^\perp, \;\; w\in \K.
\end{equation}
 By the definition of $B$, we have  $\K= \rge B$ , or equivalently,  $\K^\perp = \ker B^*$. Thus, 
we find  $q\in \R^s$ such that  $w=Bq$ and that  $ u - \nabla \psi(x) Bq \in \ker B^*$, or equivalently, $B^* u = B^*\nabla \psi(x) Bq$. 
We  claim that the  matrix $B^*\nabla \psi(x) B$ is nonsingular, which leads us to 
$$
\nabla \sigma(y)u=w=Bq=B\big(B^*\nabla \psi(x)B\big)^{-1}B^* u,
$$
and hence confirms the claimed formula for $\nabla \sigma(y)$.  
To justify our claim, pick $u\in \R^m$. We conclude from metric regularity of $G$ at $\ox$ for $0$ and and its equivalent description by condition (a)   
that  there exist $w\in \K$   and $w'\in \K^\perp$ such that $u=\nabla\psi(\ox)w+w'$.
Since $\K= \rge B$, we find  $q\in \R^s$ such that  $w=Bq$ and hence  $B^* u =B^*(\nabla\psi(\ox)Bq+w')= B^*\nabla \psi(\ox) Bq$. Since $u\in \R^m$ was taken arbitrary, the latter equality leads us 
\begin{equation*}
\rge\big(B^*\nabla \psi(\ox)B\big) = B^* \R^m = \rge B^* =\R^s,
\end{equation*}
where the last equality comes from $B$ having full column rank. This confirms  that $B^*\nabla \psi(\ox)B$ is  an $s\times s$ nonsingular matrix and completes the proof.
\end{proof}

Theorem~\ref{slcl} can be viewed as an extension of the classical inverse mapping theorem for generalized equations. This well-known result 
ensures under the nonsingularity of the Jacobian matrix that the inverse of a ${\cal C}^1$ function has a Lipschitz continuous single-valued localization, which is 
continuously differentiable. Robinson, in his landmark paper \cite{rob80}, showed that for generalized equations, one can   expect under appropriate conditions that 
their solution mappings have a Lipschitz continuous single-valued localization. Theorem~\ref{slcl} demonstrates that for nondegenerate solutions
to some particular class of generalized equations such a localization can be continuously differentiable as well.

\section*{Acknowledgments}
The first author would like to thank Vietnam Institute for Advanced Study in Mathematics for hospitality during her post-doctoral fellowship of the Institute in 2021--2022.

\end{document}